\documentclass[english,10pt]{amsart}

\usepackage{amssymb}
\usepackage{latexsym}
\usepackage{hyperref}
\usepackage{soul}
\setulcolor{red}
\sethlcolor{yellow}
\usepackage{color}

\newbox\barleftbox
\newbox\barrightbox
{\setlength\fboxsep{0pt}\setlength\fboxrule{0pt}
 \global\setbox\barleftbox=\hbox{%
   \colorbox[gray]{.8}{\hskip 4pt\relax\vrule height3pt width 0pt}%
   \hskip4pt}
 \dimen0=\ht\barleftbox \advance\dimen0-1pt
 \ht\barleftbox=\dimen0
 \dp\barleftbox=-1pt
 \global\setbox\barrightbox=\hbox{%
   \hskip4pt
   \colorbox[gray]{.8}{\hskip 4pt\relax\vrule height3pt width 0pt}}
 \dimen0=\ht\barrightbox \advance\dimen0-1pt
 \ht\barrightbox=\dimen0
 \dp\barrightbox=-1pt
}
 \usepackage[dvips]{graphicx}
\DeclareGraphicsExtensions{.pdf,.png,.jpg} 
\newcommand*{\thethmname}{}
\newtheorem{innerthm}{\thethmname}



\newcommand{\leftstrip}[1]{%
   \valign{##\cr
           \leaders\copy\barleftbox\vfill\cr
           \vbox{\hsize\marginparwidth\advance\hsize-8pt
                 \raggedright\sffamily\footnotesize #1}\cr
   }
}
\newcommand{\rightstrip}[1]{%
   \valign{##\cr
           \vbox{\hsize\marginparwidth\advance\hsize-8pt
                 \raggedright\sffamily\footnotesize #1}\cr
           \leaders\copy\barrightbox\vfill\cr
   }
}

\let\oldmarginpar\marginpar
\renewcommand\marginpar[1]{%
  \oldmarginpar[\leftstrip{#1}]{\rightstrip{#1}}}
\usepackage{amsmath}
\usepackage{amsfonts}
\usepackage{graphics}
\usepackage{color}

\newtheorem{theorem}{Theorem}[section]
\newtheorem{lemma}[theorem]{Lemma}
\newtheorem{e-proposition}[theorem]{Proposition}

\newtheorem{e-definition}[theorem]{Definition}
\newtheorem{remark}{Remark}[section]
\newtheorem{example}{Example}[section]
\newtheorem{theoreme}{Th\'eor\`eme}[section]

\newtheorem{proposition}[theoreme]{Proposition}

\newtheorem{definition}[theoreme]{Definition}

\numberwithin{equation}{section}

\long\def\salta#1{\relax}

 \def\1{\raisebox{2pt}{\rm{$\chi$}}}

\def\z{{\bf z}}
\def\rife#1{(\ref{#1})}
\newcommand{\hide}[1]{}

\newcommand{\R}{{\mathbb R}}

\newcommand{\N}{{\mathbb N}}

\renewcommand{\H}{{\mathcal H}}

\renewcommand{\epsilon}{{\varepsilon}}

\newcommand{\nada}[1]   {}

\def\div{{\rm div}}

\def\bk{\color{black}}

\def\rn{\mathbb{R}^{N}}

\def\re{\mathbb{R}}

\def\be{\begin{equation}}
\def\ee{\end{equation}}

\def\vare{\varepsilon}

\def\dive{{\rm div}}

\def\into{\int_{\Omega}}
\def\intq{\int_Q}

\def\w-1p'{W^{-1,p'}(\Omega)}

\def\w-1pd{W^{-1,p'}(D)}
\def\pw-1p'{L^{p'}(0,T;W^{-1,p'}(\Omega))}
\def\l{\textsl{L}}

\def\dys{\displaystyle}

\def\lp'n{(L^{p'}(\Omega))^{N}}

\def\lio{L^{\infty}(\Omega)}

\def\luo{L^{1}(\Omega)}
\def\lio{L^{\infty}(\Omega)}

\def\car#1{\raise2pt\hbox{$\chi$}_{#1}}

\def\lio{L^{\infty}(\Omega)}

\def\lp'n{(L^{p'}(\Omega))^{N}}

\def\t1p0{T^{1,p}_{0}(\Omega)}
\def\w-1p'{W^{-1,p'}(\Omega)}
\def\pw-1p'{L^{p'}(0,T;W^{-1,p'}(\Omega))}

\def\lil2{L^{\infty}(0,T;L^2 (\Omega))}

\def\l2h10{L^2 (0,T ; H^1_0 ( \Omega ))}

\def\m2{M^{\frac{N(p-1)}{N-1}}(\Omega)}

\def\l{\textsl{L}}

\def\al{\alpha}

\def\vare{\varepsilon}

\def\into{\int_{\Omega}}

\def\intq1{\displaystyle \int_{\Omega \times (0, 1)}}

\def\dys{\displaystyle}

\def\m{\noalign{\medskip}}

\begin{document}

\title[Large Solutions for Parabolic Equations without absorption]{Large solutions for nonlinear parabolic equations without absorption terms}

\author[S. Moll]{Salvador  Moll}

\address[S. Moll]{ Departament d'An\`{a}lisi Matem\`atica,
Universitat de Val\`encia, Valencia, Spain.}
\email{{\tt j.salvador.moll@uv.es }}

\author[F. Petitta]{Francesco Petitta}

\address[F. Petitta]{Dipartimento di Scienze di Base e Applicate
per l' Ingegneria, ``Sapienza", Universit\`a di Roma, Via Scarpa 16, 00161 Roma, Italy.}\email{francesco.petitta@sbai.uniroma1.it}


\begin{abstract}
In this paper we give a suitable notion of entropy solution of parabolic $p-$laplacian type equations with $1\leq p<2$ which blows up at the boundary of the domain. We prove existence and uniqueness of this type of solutions when the initial data is locally integrable (for $1<p<2$) or integrable (for $p=1$; i.e  the Total Variation Flow case).
\end{abstract}

\maketitle

\tableofcontents

\section{Introduction}

The study of large solutions  for elliptic partial differential equations on bounded domains  has been largely studied since the pioneering papers by Keller (\cite{ke}) and Osserman (\cite{os}); roughly speaking,  large solutions are solutions to some PDEs that go to infinity as one approaches the boundary.  Large solutions are recently come to light because of their intimate connections with some concrete applications in  Control Theory and Stochastic Processes with constraints (see for instance \cite{ll}, \cite{lepo}, and references therein).

To fix the ideas let us consider
$$
\begin{cases}
   -\Delta u  + u^q = 0& \text{in}\ \Omega,\\
 u=+\infty &\text{on}\ \partial\Omega,
  \end{cases}
$$
where $\Omega$ is a bounded subset of $\re{^N}$, $N\geq 2$. If $q>1$, then the absorption term $u^q$ allows to prove local a priori estimates on the solutions which are the key point in order  to prove existence of such solutions that explode at the boundary. This basic idea {has been} generalized in various directions and the existence of large solutions {has been} proved for a vaste amount of problems with absorption terms. As an example, concerning the $p-$laplace operator, the existence of a large solution for problem
$$
\begin{cases}
   -\Delta_p u  + u^q = 0 & \text{in}\ \Omega,\\
 u=+\infty &\text{on}\ \partial\Omega,
  \end{cases}
$$
can be proved provided $q>p-1$ (see \cite{dl}), while a large solution for problem
$$
\begin{cases}
   -\Delta_p u  +u + u|\nabla u|^q = 0& \text{in}\ \Omega,\\
 u=+\infty &\text{on}\ \partial\Omega,
  \end{cases}
$$
does exist provided $p-1<q\leq p$ (see \cite{leo}).   Observe that, the presence of the lower order absorption term is \emph{essential} in order to prove these result; the naive idea to show this fact  is that constant functions are always subsolutions for $   -\Delta_p u =0$, so that no local a priori estimates can be proved without absorption terms.

In the parabolic framework the situation is quite different since the time derivative part of the equation plays itself an absorption role. For instance, in \cite{lp}, the existence of large solutions is proved for problems whose model is
$$
\begin{cases}
    u_t-\Delta_p u + u|\nabla u|^q =0 & \text{in}\ Q_T:=(0,T)\times \Omega,\\
    u=u_0  & \text{on}\ \{0\} \times\Omega,\\
 u=+\infty &\text{on}\ (0,T)\times\partial\Omega,
  \end{cases}
$$
where $u_0$ is a locally integrable function on $\Omega$. Observe that, in this case, in contrast with the elliptic case, no zero lower order terms are needed. We also mention the paper \cite{cv} in which a theory for the so-called \emph{extended solutions} is developed  for Fast-Diffusion equations.

 If, on one hand, existence for these type of problems has been largely investigated, on the other hand   uniqueness is a harder task even in the elliptic case (see for instance \cite{nv} and references therein).

 The present paper addresses to the study of both existence and uniqueness of suitable large solutions for parabolic problems without lower order absorption terms whose model is
 \begin{equation}\label{maine}
\begin{cases}
    u_t-\Delta_p u =0 & \text{in}\ Q_T,\\
    u=u_0  & \text{on}\ \{0\} \times\Omega,\\
 u=+\infty &\text{on}\ (0,T)\times\partial\Omega,
  \end{cases}
\end{equation}
where $u_0\in L^1_{loc}(\Omega)$ is a nonnegative function and $\Omega$ is a bounded open subset of $\rn$ with {Lipschitz} boundary. In the case $p=1$ we will require a little more regularity in the space domain; namely it has to satisfy a uniform interior ball condition; i.e. there exists $s_0>0$ such that for every $x\in\Omega$ with $dist(x,\partial\Omega)<s_0$, there is $z_x\in\partial\Omega$ such that $|x-z_x|=dist(x,\partial \Omega)$ and $B(x_0,s_0)\subset\Omega$ with $x_0:=z_x+s_0\frac{x-z_x}{|x-z_x|}$.
From now on, $s_0$ will denote the radius of the uniform ball condition corresponding to $\Omega$.

As we will discuss later, if $p\geq 2$ then the absorption role (of order $1$) of $u_t$ is too weak to ensure the existence of large solutions.
On the contrary, if $1<p<2$, the possibility of proving local estimates for problems as \rife{maine} was already contained in literature as,  for instance, in \cite{lepe}. Similar arguments {were} used in \cite{BIV} to prove the existence of a so called continuous  large solution for problem \rife{maine} provided the initial data are integrable enough.  Our aim is to give a suitable notion of large solution (namely we will call {it}  \emph{Entropy Large Solution}) and to prove existence and uniqueness of {such a} solution for  problem  \rife{maine} with merely locally integrable data, also including  the case $p=1$, the Total Variation Flow case.

\medskip
There is, however, a striking difference between cases $1<p<2$ and $p=1$. Namely, if the initial data $u_0$ is bounded, 
the solutions of problem \eqref{maine} are uniformly bounded for a fixed $T$ (see Theorem \ref{theoremp1}). This feature is exclusive for the case $p=1$ and thus, we have to understand the Dirichlet condition in a relaxed way {(see condition (\ref{boundary}) in Definition \ref{Defentropy} below)}. In order to be consistent, any notion of solution of problem \eqref{maine} must satisfy $u(t)\in L^1_{loc}(\Omega)$ a.e. $t\in (0,T)$ and must
be an upper barrier to the solutions of the following  approximating problems
 \begin{equation}\label{mainen}
\begin{cases}
    (u_n)_t-\Delta_p u_n =0 & \text{in}\ Q_T,\\
    u_n=u_0  & \text{on}\ \{0\} \times\Omega,\\
 u_n=n &\text{on}\ (0,T)\times\partial\Omega.
  \end{cases}
\end{equation}

In this sense, we show in Proposition \ref{nonexistence} that, for the case $p\geq 2$, we have nonexistence of large solutions with these features.
 \section{Preliminaires and notations}\label{preliminaris}
 In this section we collect the main notation and some useful results we will use in our analysis. {We point out that most of them are only needed in the special case that $p=1$.} Since $T>0$ is fixed, for the sake of simplicity, in what follows we will use the notation $Q=Q_T$.
%

\subsection{Functions of bounded variations and some generalizations}

Let us recall that the natural energy space to study {parabolic problems related with linear growth functionals}
 is the space of functions of bounded
variation. If $\Omega$ is an open subset of $\R^N$, a
function $u \in L^1(\Omega)$ whose gradient $Du$ in the sense of
distributions is a vector valued Radon measure with finite total
mass in $\Omega$ is called a {\it function of bounded variation}.
The class of such functions will be denoted by $BV(\Omega)$ and $|Du|$ will denote the total variation of the measure $Du$.

Moreover,  an ${\mathcal L}^N$-measurable subset $E$ of $\R^N$ has finite perimeter if $\chi_E\in BV(\R^N)$. The perimeter of $E$ is defined by $Per(E) =|D\chi_E|$

For further information and properties concerning functions of bounded variation
we refer to \cite{Ambrosio}, \cite{EG} or \cite{Ziemer}.

\medskip

 We need to consider the following truncature functions.
For $a < b$, let $T_{a,b}(r) := \max(\min(b,r),a)$. As usual, we
denote $T_k = T_{-k, k}$.
Given any function $u$ and $a,b\in\R$ we shall use the notation
$[u\geq a] = \{x\in \R^N: u(x)\geq a\}$, $[a \leq u\leq b] =
\{x\in \R^N: a \leq u(x)\leq b\}$, and similarly for the sets
 $[u
> a]$, $[u \leq a]$, $[u < a]$, etc.

\medskip
Given a real function $f(s)$, we define  its  positive and negative part as, respectively, $f^+ (s)=\max(0,f(s))$ and $f^- (s)=\min(0,f(s))$. We consider the set of truncatures $\mathcal P$ of all nondecreasing continuous
functions $p: \R\to\R$, such that there exists $p'$ except a finite set and $supp( p')$
is compact.
{For our purposes, }we  need to consider the function spaces $$TBV(\Omega):= \left\{
u \in L^1(\Omega)^+  \ :  \ \ T_k(u) \in BV(\Omega), \ \ \forall \ k>0 \right\},$$ {$$TBV_{loc}(\Omega):= \left\{
u \in L^1_{loc}(\Omega)^+  \ :  \ \ T_k(u) \in BV(\Omega), \ \ \forall \ k>0 \right\},$$ }
\noindent and to give a sense to the
Radon-Nikodym derivative (with respect to the Lebesgue measure)
$\nabla u$ of $Du$ for a function $u \in TBV_{loc}(\Omega)$. Using
chain's rule for BV-functions (see for instance \cite{Ambrosio}),
with a similar proof to the one given in Lemma 2.1 of
\cite{B6}, we obtain the following result.

\begin{lemma}\label{WRN}
For every $u \in TBV_{loc}(\Omega)$ there exists a unique measurable
function $v : \Omega \rightarrow \R^N$ such that
\begin{equation}\label{E1WRN}
\nabla T_k(u) = v \1_{[| u | < k]} \ \ \ \ \ {\mathcal
L}^N-{\rm a.e.}, \ \ \forall \ k>0.
\end{equation}
\end{lemma}

Thanks to this result we define $\nabla u$ for a function $u \in TBV_{loc}(\Omega)$ as the
unique function $v$ which satisfies (\ref{E1WRN}). {Obviously, if $w\in W_{\rm loc}^{1,1}(\Omega)$, then the generalized gradient turns out to coincide with the classical distributional one}. This notation will be used throughout
in the sequel.

\medskip

We recall the following result (\cite{ACM4:01}, Lemma 2).

\begin{lemma}\label{CR}
If $u \in TBV(\Omega)$, then $p(u) \in BV(\Omega)$ for every $p \in {\mathcal P}$. Moreover, $\nabla
p(u) = p^{\prime}(u) \nabla u$ \ ${\mathcal L}^N$-{\rm a.e.}
\end{lemma}

\subsection{A generalized Green's formula}\label{GreenAnz}

We shall need several results from \cite{Anzellotti1} (see also \cite{ACMBook}). Let
$$ X(\Omega) = \left\{ \z \in L^{\infty}(\Omega; \R^N) \
:
 \ \div(\z)\in L^1(\Omega) \right\}
$$

 If $\z \in X(\Omega)$ and $w \in BV(\Omega) \cap
L^{\infty}(\Omega)$
we define the functional $(\z,Dw):
C^{\infty}_{0}(\Omega) \rightarrow \R$ by the formula
\begin{equation}\label{defmeasx1}
\langle (\z,Dw),\varphi\rangle := - \int_{\Omega} w \, \varphi \,
\div(\z) \, dx - \int_{\Omega} w \, \z \cdot \nabla \varphi \, dx.
\end{equation}
In \cite{Anzellotti1} it is proved that $(\z,Dw)$ is a Radon measure in $\Omega$  verifying
$$
\int_{\Omega} (\z,Dw) = \int_{\Omega} \z \cdot \nabla w \, dx \ \
\ \ \ \forall \ w \in W^{1,1}(\Omega) \cap L^{\infty}(\Omega).
$$

Moreover, for all $w\in BV(\Omega)\cap L^\infty(\Omega)$,  $(\z,Dw)$ is absolutely continuous with respect to the total variation of $w$ and it holds,
\begin{equation}\label{acota}
\int_{\Omega} (\z,Dw) \leq \|z\|_{\lio}\into |Dw|.
\end{equation}

In \cite{Anzellotti1}, a weak trace on $\partial \Omega$ of the
normal component of  $\z \in X(\Omega)$ is defined.
Concretely, it is proved that there exists a linear operator \
$\gamma : X(\Omega) \rightarrow L^{\infty}(\partial \Omega)$
such that
$$\Vert \gamma(\z) \Vert_{L^{\infty}(\partial\Omega)} \leq \Vert \z \Vert_{L^{\infty}(\Omega; \R^N)}, $$
where $$Ê\gamma(\z) (x) = \z(x) \cdot \nu(x) \ \ \ \ {\rm for \ all} \ x \in
\partial
\Omega \ \ {\rm if} \ \ \z \in C^1(\overline{\Omega}; \R^N).$$ We
shall denote \ $\gamma (\z)(x)$ by $[\z, \nu](x)$. Moreover, the
following {\it Green's formula}, relating the function $[\z, \nu]$
and the measure $(\z, Dw)$, for  $\z \in X(\Omega)$ and  $w \in
BV(\Omega) \cap L^{\infty}(\Omega)$ is established
\begin{equation}\label{Green}
\int_{\Omega} w \ \div (\z) \ dx + \int_{\Omega} (\z, Dw) =
\int_{\partial \Omega} [\z, \nu] w \ d\H^{N-1}.
\end{equation}

\medskip

To make precise our notion of solution we also need to recall the
following definitions given in \cite{ACM4:01}.

\medskip

 We define
the space
$$Z(\Omega):= \left\{ (\z, \xi) \in L^{\infty} (\Omega; \R^N) \times
BV(\Omega)^* \ : \ \ {\rm div}(\z) = \xi \ \ {\rm in} \ \ {\mathcal
D}^{\prime}(\Omega)\right\}.$$ We set $R(\Omega):=
W^{1,1}(\Omega) \cap L^{\infty}(\Omega) \cap C(\Omega)$. For \
$(\z, \xi) \in Z(\Omega)$ and $w \in R(\Omega)$ we define
$$\langle (\z, \xi), w \rangle_{\partial \Omega}:= \langle \xi, w
\rangle_{BV(\Omega)^*,BV(\Omega)} + \int_{\Omega} \z \cdot \nabla
w \ dx.$$ Then, working as in the proof of Theorem 1.1. of
\cite{Anzellotti1}, we obtain that if $w , v \in R(\Omega)$ and $w
= v$  on $\partial \Omega$ one has
\begin{equation}\label{e1EU}
\langle (\z, \xi), w \rangle_{\partial \Omega} = \langle (\z,
\xi), v \rangle_{\partial \Omega} \ \ \ \ \ \forall \ (\z, \xi)
\in Z(\Omega).
\end{equation}
As a consequence of (\ref{e1EU}), we can give the following
definition.
Given \ $u \in BV(\Omega)\cap L^{\infty}(\Omega)$ and
\ $(\z, \xi) \in Z(\Omega)$, we define \ $\langle (\z, \xi), u
\rangle_{\partial \Omega}$ by setting
$$\langle (\z, \xi), u \rangle_{\partial \Omega}:= \langle (\z, \xi), w
\rangle_{\partial \Omega},$$ where $w$ is any function in
$R(\Omega)$ such that $w = u$ on $\partial \Omega$.
In
\cite{ACM4:01} it is defined a weak trace on $\partial\Omega$ of the normal component of $(\z,\xi)$ which we will denote as $[\z, \nu](x)$.

\subsection{The space $\big(L^1(0, T; BV(\Omega)_2)\big)^*$.}

We need to consider the space  $BV(\Omega)_2$, defined as
$BV(\Omega) \cap L^2(\Omega)$ endowed with the norm
$$\Vert v \Vert_{BV(\Omega)_2}: = \Vert v \Vert_{L^2(\Omega)} + \vert
Dv \vert (\Omega).$$ It is easy to see that $L^2(\Omega) \subset
BV(\Omega)_2^*$ and
\begin{equation}\label{e3EU}
\Vert v \Vert_{BV(\Omega)_2^*} \leq \Vert v \Vert_{L^2(\Omega)} \
\ \ \ \ \ \ \ \forall \ v \in L^2(\Omega).
\end{equation}

It is well known (see \cite{Schwartz}) that the dual space
$\big(L^1(0, T; BV(\Omega)_2)\big)^*$ is isometric to the space
$L^{\infty}(0, T; BV(\Omega)_2^*, BV(\Omega)_2)$ of all weakly$^*$
measurable functions $f : [0, T] \rightarrow BV(\Omega)_2^*$, such
that \ $v(f) \in L^{\infty}(]0, T[)$,  where $v(f)$ denotes the
supremum of the set $\{ \vert \langle v, f \rangle_{BV(\Omega)_2, BV(\Omega)_2^*} \vert \ : \
\Vert v \Vert_{BV(\Omega)_2} \leq 1 \}$ in the vector lattice of
measurable real functions. Moreover, the dual pairing of the
isometry is defined by
$$\langle v, f \rangle = \int_0^T \langle v(t), f(t) \rangle_{BV(\Omega)_2, BV(\Omega)_2^*} \ dt,$$
for $v \in L^1(0, T; BV(\Omega)_2)$ and $f \in L^{\infty}(0, T;
BV(\Omega)_2^*, BV(\Omega)_2)$.

\medskip

By $L^1_{w}(0,T;BV(\Omega))$ we denote the space of weakly
measurable functions $v:[0,T] \to BV(\Omega)$ (i.e., $t \in [0,T]
\to \langle v(t),\phi \rangle$ is measurable for every $\phi \in
BV(\Omega)^*$) such that $\int_0^T \| v(t)\|_{BV(\Omega)} \, dt< \infty$.
Observe that, since $BV(\Omega)$ has a separable predual (see
\cite{Ambrosio}), it follows easily that the map $t \in [0,T]\to
\| v(t) \|_{BV(\Omega)}$ is measurable. By  $L^1_{loc, w}(0, T;
BV(\Omega))$ we denote the space of weakly measurable functions
$v:[0,T] \to BV(\Omega)$ such that the map $t \in [0,T]\to |
v(t) |$ is in $L^1_{loc}(]0, T[)$

\smallskip

Let us recall the  following definitions given in \cite{ACM4:01}.

\begin{definition}\label{Def2}
Let $\xi \in \big(L^{1}(0,T;BV(\Omega)_2) \big)^*$. We say that
$\xi$ is {\it the time derivative} in the space
$\big(L^{1}(0,T;BV(\Omega)_2 \big)^*$ of a function $u\in
L^1((0,T) \times \Omega)$ if
$$\int_0^T \langle \Psi(t),\xi(t) \rangle_{BV(\Omega)_2, BV(\Omega)_2^*} dt = - \int_0^T \int_{\Omega}
u(t,x) \Theta(t,x) dx dt$$ for all test functions $\Psi \in
L^1(0,T;BV(\Omega)_2)$ with compact support in time, such that there
exists $\Theta \in L^1_{w}(0,T;BV(\Omega)) \cap L^{\infty}(Q)$
with $\Psi(t) =\displaystyle \int_0^t \Theta(s) ds$, the integral
being taken as a Pettis integral (\cite{DU}).
\end{definition}

Note that if  $w \in L^1(0, T;BV(\Omega)) \cap L^{\infty}(Q)$
and $\z \in L^{\infty}(Q; \R^N)$ is such that there exists $\xi
\in \big(L^{1}(0,T;BV(\Omega) \big)^*$ with ${\rm div}(\z) = \xi$
in ${\mathcal D}{'}(Q)$, we can define, associated to the pair
$(\z, \xi)$, the distribution $(\z, Dw)$ in $Q$ by
\begin{equation}\label{distribucion}
\begin{array}{ll}
\displaystyle\langle (\z, Dw), \phi \rangle :=  - \displaystyle
\int_0^T \langle  w(t) \phi(t),\xi(t) \rangle _{BV(\Omega)_2, BV(\Omega)_2^*}\, dt \\ \\ \dys- \int_0^T
\int_{\Omega} \z(t, x) w(t, x) \nabla_x \phi(t, x) \, dxdt.
\end{array}
\end{equation}
for all $\phi \in {\mathcal D}(Q)$.

\begin{definition}\label{Def3} {\rm Let $\xi\in \big(L^{1}(0,T;BV(\Omega)_2) \big)^*$ and  $\z \in
L^{\infty}(Q; \R^N)$. We say that $\xi = {\rm div}(\z)$ in
$\big(L^{1}(0,T;BV(\Omega)_2 )\big)^*$ if $(\z,Dw)$
 is a Radon measure in $Q$ with
normal boundary values $[\z,\nu]\in L^{\infty}((0,T)\times
\partial \Omega)$, such that
$$
\int_{Q} (\z,Dw) + \int_0^T \langle w(t),\xi(t) \rangle_{BV(\Omega)_2,BV(\Omega)_2^*} dt =
\int_0^T \int_{\partial \Omega} [\z(t,x),\nu] w(t,x) d\H^{N-1} dt,
$$
for all $w \in L^1(0,T;BV(\Omega)) \cap L^{\infty}(Q)$.}
\end{definition}

Finally, throughout the paper  $\omega(\nu,\vare,n,k)$ will  indicate any quantity that vanishes as the parameters go to
their (obvious, if not explicitly stressed) limit point with the same order in which they appear,  that is, for example
\[
\dys\lim_{\nu\rightarrow 0}\limsup_{n\rightarrow +\infty}
\limsup_{\vare\rightarrow 0} |\omega(\vare,n,\nu)|=0.
\]

\section{The case $1<p<2$. Existence and uniqueness}

As we said, to deal with {proving} both existence and uniqueness of solutions to problem
\begin{equation}\label{renpb}
\begin{cases}
    u_t-\Delta_p u =0 & \text{in}\ Q,\\
    u=u_0  & \text{on}\ \{0\} \times\Omega,\\
 u=+\infty &\text{on}\ (0,T)\times\partial\Omega,
  \end{cases}
\end{equation}
where  $u_0\in L^1_{loc}(\Omega)$ is a nonnegative function and $1\leq p<2$, we need to introduce a suitable notion of solution. We choose  an Entropy/Renormalized type notion that allows us to treat in a unifying way both the case $1<p<2$ and $p=1$.
 The entropy formulation  is  nowadays the usual one in order to deal
with both existence and uniqueness of \emph{infinite energy} solutions for nonlinear PDEs (see for instance \cite{BlMu}, \cite{dpp} and \cite{pe}). We also have to specify how the boundary datum $+\infty$ is attained. We do that in a very weak sense that is, roughly speaking, we ask the truncations $T_k(u)$ of the solution to belong to the energy space with trace $k$, for any $k>0$. In what follows $W^{1,p}_k (\Omega)$ will denote the subspace of $W^{1,p}(\Omega)$ of those functions whose trace at the boundary is $k$.

%
\medskip
Let us fix $1<p<2$. Here is our definition of Entropy Large Solution for problem \rife{renpb}:

\begin{definition}\label{deflargeentropy}A measurable function $u:Q\to \R$ is an entropy solution of \eqref{renpb} (also an entropy large solution) if $u\in C(0,T;L^1_{loc}(\Omega))$, $T_k(u)\in L^p(0,T;W^{1,p}_k(\Omega))$
, for all $k>0$, $|\nabla u|^{p-1}\in L^1_{loc}(Q)$,  and $$\int_{Q} \eta |\nabla u|^{p-2}\nabla u\cdot \nabla S(T_h(u)-T_h(l))+\int_{Q} S(T_h(u)-T_h(l))|\nabla u|^{p-2}\nabla u\cdot\nabla \eta $$$$=\int_{Q} j_{S,h,l}(u)\eta_t$$
for all $S\in\mathcal P$, $h>0$ $l\in\R$, $\eta\in \mathcal D(Q)$ and $j_{S,h,l}(r):=\int_l^r S(T_h(s)-T_h(l))\, ds$. Moreover, $u(0)=u_0$ in $L^1_{loc}(\Omega)$.
\end{definition}
This formulation ressembles the one given in \cite{ACM} for solutions of the Dirichlet problem corresponding to the Total Variation Flow. However, we have to take into account another truncation $T_h$ in the definition since solutions may blow up at the boundary. Observe that, thanks to the regularity assumptions on $u$ all the terms in the previous definition make sense; moreover,  an entropy large solution turns out to be a distributional solution as next result shows.
\begin{proposition}
An entropy large solution is a distributional solution for problem \rife{renpb}.
\end{proposition}
\begin{proof}
Let us introduce the following auxiliary function:
\be\label{sgei}
S_j (\tau)=\int_0^\tau [1-T_1 (G_j (s))] ds\,,
\ee
where $G_k (s)= s-T_k (s)$. Now consider $S=S_j$ and $l=0$ in Definition \ref{deflargeentropy}, to get
$$\int_{Q} \eta |\nabla u|^{p-2}\nabla u\cdot \nabla S'_j (T_h(u))+\int_{Q} S'_j(T_h(u))|\nabla u|^{p-2}\nabla u\cdot\nabla \eta $$$$=\int_{Q} S_j (T_h(u))\eta_t\,.$$
Now observe that, since $|\nabla u|^{p-1}\in L^1_{loc}(Q)$, the first term in the above  equality tends to zero as $j$ goes to infinity. Finally,  for $j>h$ we have  $S_j (T_h(u))=u$, so we can let $j$ tend to infinity to get
$$\int_{Q} |\nabla u|^{p-2}\nabla u\cdot\nabla \eta =\int_{Q} u\eta_t\,,$$
for any $\eta \in  \mathcal D(Q)$.
\end{proof}

Our main result is the following
\begin{theorem}\label{maint}
Let $u_0\in L^1_{loc}(\Omega)$ and $1<p<2$. Then, there exists a unique entropy large solution for problem \rife{renpb}.
\end{theorem}

Let us collect some useful tools we are going to use later on. We start with the following definition

\begin{definition}[\cite{BIV}]A continuous large solution of \eqref{renpb} is a continuous function $u\in C^2_{loc}(0,T;L^2_{loc}(\Omega))\cap L^p_{loc}(0,T;W^{1,p}_{loc}(\Omega))$ such that it takes the boundary data in the continuous sense (i.e. $u(x,t)\to +\infty$ as $x\to\partial\Omega$, for any fixed $t\in (0,T)$) and verifies $$\int_K u(t_2)\varphi(t_2)\,dx-\int_K u(t_1)\varphi(t_1)\,dx+\int_{t_1}^{t_2}\int_K (-u\varphi_t+|\nabla u|^{p-2}\nabla u\cdot\nabla \varphi)\,dx\,dt=0$$ for every open bounded set $K\subset\subset\Omega$, for every $[t_1,t_2]\subset [0,T]$ and for any test function $\varphi\in W^{1,2}_{loc}(0,T;L^2(K))\cap L^p_{loc}(0,T;W_{0}^{1,p}(K))$.
\end{definition}

\begin{theorem}\label{biv}[Theorems 4.1 and 4.2 in \cite{BIV}] Let $1<p\leq p_c:=\frac{2N}{N+1}$ and $r>\frac{n(2-p)}{p}$ or $p_c<p<2$ and $r\geq 1$. Given $u_0\in L^r_{loc}(\Omega)$, there exists a continuous large solution $u$ of \eqref{renpb}. Such solutions are moreover H\"older continuous in the interior {and they verify $$\frac{C_0t^\frac{1}{2-p}}{{\rm dist}(\partial\Omega)^\frac{p}{2-p}}\leq u(t,x)\leq \frac{C_1t^\frac{1}{2-p}}{{\rm dist}(\partial\Omega)^\frac{p}{2-p}}+C_2\qquad \forall (t,x)\in Q\,,$$for some positive constants $C_0,C_1,C_2$}.
\end{theorem}

\subsection{Existence of entropy large solutions}

Our argument will be by approximation. Consider problem
\begin{equation}\label{senzah}
\begin{cases}
    (u_n)_{t}-\Delta_p u_n =0 & \text{in}\ Q,\\
    u_n=u_{0n}  & \text{on}\ \{0\} \times \Omega,\\
 u_n=n &\text{on}\ (0,T)\times\partial\Omega\,,
  \end{cases}
\end{equation}
where $u_{0n}=T_n(u_0)$.
 The existence of such a solution is classical (see for instance \cite{dib, BIV}).  In order to pass to the limit, we need first to check that such a solution coincides with an entropy one, that is it satisfies the following
 \begin{definition}
  \label{defentropy3}A measurable function $u_n:Q\to \R$ is an entropy solution of \rife{senzah} if $u_n\in C(0,T;L^1(\Omega))$, $T_k(u_n)\in L^p(0,T;W^{1,p}_n(\Omega))$, for all $k>0$ and $$\int_{Q} \eta |\nabla u_n|^{p-2}\nabla u_n\cdot \nabla S(T_h(u_n)-T_h(l))+\int_{Q} S(T_h(u_n)-T_h(l))|\nabla u_n|^{p-2}\nabla u_n\cdot\nabla \eta $$$$=\int_{Q} j_{S,h,l}(u_n)\eta_t\,,$$ for all $S$, $h>0$ $l\in\R$, $\eta\in \mathcal D(Q)$ and $j_{T,h,l}(r):=\int_l^r S(T_h(s)-T_h(l))\, ds$. Moreover, $u_n(0)=u_{0n}$.
\end{definition}

The next technical result will be proved by the use of a suitable localized Steklov-type regularization argument.
\begin{proposition}
Let $w$ be a weak solution for problem \rife{senzah}, then $w$ is an entropy solution of the same problem in the sense of Definition \ref{defentropy3}.
\end{proposition}
\begin{proof}
For a given nonnegative function $\eta\in \mathcal{D}(Q)$  let us define
$$
w^\tau (t)=\frac{1}{\tau}\int_{t-\tau}^t \eta(s) S(T_h(w(s))-T_h(l))\ ds
$$
for $\tau>0$ small enough such that the previous expression makes sense. The argument can be easily made rigorous by extending the function as  $0$  in $(-\delta, 0)\cup(T, T+\delta)$ for $\delta>0$.
 It is not difficult to check that
$$
w^\tau (t)\longrightarrow \eta(t) S(T_h(w(t))-T_h(l))\ \ \ \text{as $\tau\to 0$},
$$
in $ L^p(0,T;W^{1,p}_0(\Omega))$ (see \cite{cw} for further details).

Therefore, let us choose $w^\tau$ as test function in the problem solved by $w$, that is
$$
\int_0^T\langle w_t, w^\tau\rangle dt=\int_Q |\nabla w|^{p-2}\nabla w\cdot \nabla w^\tau\ dxdt.
$$
First of all, thanks to the properties of $w^\tau$ we readily have
$$
\lim_{\tau\to 0}\int_Q |\nabla w|^{p-2}\nabla w\cdot \nabla w^\tau\ dxdt=\lim_{\tau\to 0} \int_0^T\langle w_t, w^\tau\rangle dt$$
$$= \int_{Q} \eta |\nabla w|^{p-2}\nabla w\cdot \nabla S(T_h(w)-T_h(l))\ dxdt +\int_{Q} S(T_h(w)-T_h(l))|\nabla w|^{p-2}\nabla w\cdot\nabla \eta\ dxdt\,.
$$
On the other hand,
$$
\begin{array}{l}
\dys \int_0^T\langle w_t, w^{\tau}\rangle dt= - \int_Q w^{\tau}_{t} w \\\\
\dys = -\frac{1}{\tau}\int_Q (\eta(t) S(T_h(w(t))-T_h(l))- \eta(t-\tau) S(T_h(w(t-\tau))-T_h(l)))w(t)\ dxdt\\\\
\dys =  \frac{1}{\tau}\int_Q \eta(t) S(T_h(w(t))-T_h(l))(w(t)-w(t+\tau))\ dxdt\\\\
\dys\leq \frac{1}{\tau}\int_Q \eta(t)\left(\int_{w(t)}^{w(t+\tau)} S(T_h(s)-T_h(l)) ds\right)dxdt\,\\\\
=\dys \frac{1}{\tau}\int_Q \eta(t) (j_{S,h,l}(w(t+\tau) )-j_{S,h,l}(w(t)))=\frac{1}{\tau}\int_Q  \frac{(\eta(t+\tau) -\eta(t))}{\tau}j_{S,h,l}(w(t))
\end{array}
$$
where, to get the inequality, we have used the fact that $S$ is nondecreasing. Passing to the limit with respect to $\tau$ we get
$$
\limsup_{\tau\to 0} \int_0^T\langle w_t, w^{\tau}\rangle dt \leq  \int_Q  j_{S,h,l}(w)\eta_t\,.
$$
Analogously, we can fix $\tau>0$ and define
$$
{\tilde w^\tau} (t)=\frac{1}{\tau}\int_{t}^{t+{\tau}} \eta(s) S(T_h(w(s))-T_h(l))\ ds,
$$
to obtain, reasoning in a similar way, that
$$
\liminf_{\tau\to 0} \int_0^T\langle w_t,{\tilde w^\tau}\rangle dt \geq  \int_Q  j_{S,h,l}(w)\eta_t\,.
$$
Gathering together all the previous results we finally get,
$$\int_{Q} \eta |\nabla w|^{p-2}\nabla w\cdot \nabla S(T_h(w)-T_h(l))+\int_{Q} S(T_h(w)-T_h(l))|\nabla w|^{p-2}\nabla w\cdot\nabla \eta $$$$=\int_{Q} j_{S,h,l}(w)\eta_t\,,$$
\end{proof}
\bk
 We will also need the following proposition that contains the key local estimates we shall use as well as a global estimate on $T_k (u_n)$ that is essential in order to prove that the boundary data is attained.
\begin{proposition}\label{pro}
Let $u_n$ be a sequence of solutions
 of problem \rife{senzah} and $u_{0,n}\geq 0$ verifying $u_{0,n+1}\geq u_{0,n}$ and strongly converging to $u_0$ in $L^1(\Omega)$. Then,
\begin{equation}\label{pro1}
\|u_n\|_{L^\infty (0,T; L^1_{loc} (\Omega))}\leq C,
\end{equation}
and
\begin{equation}\label{pro2}
\int_Q|\nabla T_k (u_n)|^p\ dxdt\leq Ck \qquad \forall k>0.
\end{equation}
Moreover, there exists a measurable function $u$ such that $T_k (u)\in L^p(0,T;W^{1,p}_k(\Omega))$ for any $k>0$, $u\in L^\infty (0,T; L^1_{loc} (\Omega))$ { verifying \begin{equation}
   \label{lowerbd}\frac{C_0t^\frac{1}{2-p}}{{\rm dist}(\partial\Omega)^\frac{p}{2-p}}\leq u(x,t)
 \end{equation}for some positive constant $C_0$} and, up to  a subsequence,  we have
\be\label{conver}
\begin{split}
& u_n \rightarrow u \qquad \hbox{a.e. in $Q$    and strongly in $L^1_{loc} (Q)$,}\\
&T_k (u_n)\rightharpoonup T_k (u)\qquad \hbox{weakly in $ L^p(0,T;W^{1,p}_k(\Omega))$  and a.e. in $Q$,}\\
&\nabla u_n\rightarrow \nabla u \qquad \hbox{a.e. in $Q$.}\\
&|\nabla u_n|^{p-2}\nabla u_n\rightarrow |\nabla u|^{p-2}\nabla u\qquad  \hbox{ in $L^1_{loc}(Q)$.}
\end{split}
\ee
Finally, $u\in C([0,T]; L^1_{loc}(\Omega))$ and $u(0)=u_0$.
\end{proposition}

\begin{proof}
The \emph{global} estimate \eqref{pro2} can be proved as in \cite{lp} by taking $T_k(u_n) - k $ as test function in \eqref{senzah}.  Moreover,  local estimate \rife{pro1} is proved in \cite{lepe} as well as   the convergence results that are a consequence of a result by Simon (\cite{si}).    In particular,  there exists a function $u$ such that  $T_k (u_n)\to T_k (u)$ {weakly in} $ L^p (0,T;W^{1,p}_k(\Omega))$.

{ For the lower bound \eqref{lowerbd}, we compare $u_n$ with the solution (call it $v_n$) of \rife{senzah} with $0$ as initial data. By comparison principle, we have that \begin{equation}\label{compp0}u_n\geq v_n \,,\quad {\rm in \ } Q.\end{equation}
Note also that $u_n$ and $v_n$ are subsolutions of the problems \rife{senzah} with boundary datum $n+1$ and, respectively, $u_{0,n}$ and $0$ as initial datum. Therefore, by monotonicity and \eqref{compp0} we have that $u\geq v$ with $v$ being the continuous large solution of the $p-$laplacian with $0$ as initial data. By Theorem \ref{biv} we finally have that $$u\geq v\geq\frac{C_0t^\frac{1}{2-p}}{{\rm dist}(\partial\Omega)^\frac{p}{2-p}}$$for some positive constant $C_0$.}
The last assertion is also proved in \cite{lepe}.
\end{proof}

Previous result contains almost all the tools in order to pass to the limit in \rife{defentropy3}. The last tool is given by the following
\begin{lemma}\label{fuerte}
Let $u_n $ and $u$ be as in Proposition \ref{pro}. Then
$$
T_k (u_n) \longrightarrow T_k (u)\ \ \ \text{strongly in } \ \ L^p(0,T; W^{1,p}_{k}(\Omega))Ê
$$
for all $k>0$.
\end{lemma}
\begin{proof} A deeper look at the proof of \rife{lowerbd} shows that it is sufficient to prove that the convergence holds in the space $L^p(0,T; W^{1,p}_{loc}(\Omega))$.
The proof of this result runs exactly as the one in the  proof of Theorem 2.3 in \cite{lp} with some simplifications. For the convenience of the reader we sketch here the main steps of the proof by highlighting the differences with the proof in \cite{lp}.  Without loss of generality we can prove that $T_k (u_n) \to T_k (u)$ strongly in $L^p(0,T; W^{1,p}(B_R))Ê $ for any ball $B_R\subset\subset\Omega$. So,  let us fix two positive numbers $R$ and $\rho$ such that the ball  $B_{R+\rho}$ is contained in $\Omega$ and let us  consider a cut-off  function $\xi =\xi^{\rho}_R (|x|)$ in  $C^{1}_0 (\Omega)$  such that
\be\label{xi}
\left\{
\begin{array}{ll}
\xi \equiv 1 \quad &\mbox{ in   }\, B_R\\[1.5 ex]
0< \xi<1 \quad &\mbox{ in   }\, B_{R+\rho}\slash B_R \\[1.5 ex]
\xi \equiv 0 \quad & \mbox{ in  } \, \Omega\slash B_{R+\rho}\,.
\end{array}
\right.
\ee
\emph{Step  $1$ : Estimate on the strips. }
\be\label{uniform}
\dys \liminf_{j\to \infty} \, \sup_{n\in \mathbb{N}} \,
\dys \int_{ [ j \leq  u_n\leq  j+1]} |\nabla u_n|^{p}     \xi=0\,.
\ee
The proof of this fact is obtained by taking $T_1(G_j(s))\xi$ (recall $G_k(s) = s -T_k(s)$) as test function in \rife{senzah} and using the estimates in Proposition \ref{pro}.

\medskip
\noindent\emph{Step $2$: Landes Regularization.}  We would like to take $T_k(u)$ as test function in \rife{senzah}. This is not possible in general and we use a regularization argument introduced in \cite{lan}. We consider the Landes regularization  of $T_k (u)$ dependent on the parameter $\nu>0$ and we denote it by $T_k(u)_\nu$, the main features of this functions being
$$
\frac{d T_k (u)}{dt}=\nu(T_k (u)-T_k (u)_\nu)
$$
in the sense of distributions, and
$$
\begin{array}{l}
\dys T_k (u)_\nu\longrightarrow T_k (u)\ \ \ \ \text{strongly in $L^p(0,T;W^{1,p}_{loc}(\Omega)$) and a.e. in $Q$},\\\\
\dys\|T_k (u)_\nu\|_{L^\infty (Q)}\leq k \ \ \ \ \forall\nu>0\,,
\end{array}
$$
(see for instance \cite{po,pe} for more details).

\medskip\medskip
Now we take $\psi=(T_k (u_n)-T_k (u)_{ \nu})S_j  '(u_n) \xi  $ as test function  in \rife{senzah}, where    $S_j (s)$ is {as} in \rife{sgei} to obtain

\be\label{nabo}
\begin{array}{c}
\dys \int_0^T  \langle  S_j (u_n)_t  \,,\,  (T_k (u_n) -T_k (u)_{\nu} ) \xi \rangle\\\\
\dys+ \int_Q |\nabla u_n|^{p-2}\nabla u_n\cdot  \nabla \xi \, (T_k (u_n)-T_k (u)_{ \nu})  S'_j (u_n)
\\\\
+ \dys \int_Q |\nabla u_n|^{p-2}\nabla u_n\cdot  \nabla  \Big(T_k (u_n)-T_k (u)_{ \nu}\Big)  \,  S_j '(u_n) \xi
\\\\
- \dys\int_{[j\leq u_n< j+1]} |\nabla u_n|^{p}\,  (T_k (u_n)-T_k (u)_{ \nu})  \xi
=0.
\end{array}
\ee

\medskip \medskip
Now we look at the four above integrals one by one:
\medskip \medskip

\noindent\emph{Step $3$: Time Derivative Part. } We have  that, for $j$ large enough
\be \label{tempo}
\int_0^t \langle S_j (u_n)_t \, , \, (T_k (u_n)-T_k (u)_{\nu} )  \xi \rangle \geq \omega (n,\nu) \,,
\ee
being the proof of this fact identical to the one in \cite{lp}.

\medskip \medskip

\noindent\emph{Step $4$: The second integral in \rife{nabo}.} Using Proposition \ref{pro} and the definition of $T_k (u)_\nu$, we get
$$
 \int_Q |\nabla u_n|^{p-2}\nabla u_n\cdot  \nabla \xi \, (T_k (u_n)-T_k (u)_{ \nu})  S'_j (u_n)=\omega (n,\nu)\,.
$$
\medskip \medskip

\noindent\emph{Step $5$: The Energy on the strips:} Using the boundedness of $(T_k (u_n)-T_k (u)_{ \nu})$ and \rife{uniform} we get
$$
\dys\int_{[j\leq u_n< j+1]} |\nabla u_n|^{p}\,  (T_k (u_n)-T_k (u)_{ \nu})  \xi =\omega(j),
$$
uniformly with respect to $n$ and $\nu$.

\medskip \medskip

\noindent\emph{Step $6$: Last Step:} We notice that
$$
\int_Q   | \nabla T_k (u)_{\nu} |^{p-2}\nabla T_k(u)_\nu \cdot  \nabla  (T_k (u_n)-T_k (u)_{ \nu})
S_j '(u_n) \xi =  \omega (n,\nu),
$$
so that we can add this term in \rife{nabo} to obtain, gathering together all the previous steps
 $$
\begin{array}{c}
\dys \int_Q  (|\nabla T_k (u_n) |^{p-2}\nabla T_k(u_n) - | \nabla T_k (u)_{\nu} |^{p-2}\nabla T_k(u)_\nu) \cdot  \nabla  (T_k (u_n)-T_k (u)_{ \nu})
S_j '(u_n) \xi  \\\\ \dys
\leq  \omega (n,\nu)+\omega (j) \,,
\end{array}
$$
that yields, thanks to a classical monotonicity argument (see Lemma 5 in \cite{bmp}),
\be\label{convtro}
T_k (u_n) \to T_k (u)  \quad \mbox{strongly in } L^p (0,T; W^{1,p} (B_R )) \,.
\ee

\end{proof}

\begin{theorem}\label{esiste}
There exists an entropy large solution for problem \rife{renpb}.	
\end{theorem}
\begin{proof}
We want to pass to the limit in the equation of Definition \ref{defentropy3}. Thanks to Proposition \ref{pro} this is an easy task for each term but the first one. Observe that,
$$| \eta |\nabla u_n|^{p-2}\nabla u_n\cdot \nabla S(T_h(u_n)-T_h(l)) |\leq \eta |\nabla T_h(u_n)|^{p} S'(T_h(u_n)-T_h(l)),
$$

 So, we can  use Lemma \ref{fuerte} and Vitali's theorem to see that
$$\begin{array}{l}\dys \int_{Q} \eta |\nabla u_n|^{p-2}\nabla u_n\cdot \nabla S(T_h(u_n)-T_h(l))\\\\ \dys =\dys \int_{Q} \eta |\nabla u|^{p-2}\nabla u\cdot \nabla S(T_h(u)-T_h(l))+\omega(n),
\end{array}
$$
and conclude that $u$ satisfies the integral formulation in Definition \ref{deflargeentropy}.

Observe that, since by Proposition  \ref{pro} we have  $T_k(u)\in L^p(0,T;W^{1,p}_k(\Omega))$ and  $u\in C([0,T]; L^1_{loc}(\Omega))$. Therefore,  both the boundary and initial  data are attained.

\end{proof}

\salta{
Next result is a nonexistence result of large solutions in the sense given in the Introduction
that gives a justification to the assumption on $p$.
\begin{proposition}\label{nonexistence}
Let $p\geq 2$ and let $u_n (x,t)$ be the approximate solutions defined by \rife{senzah}.  Then, $\|u_n(t_0)\|_{L^1_{loc}(\Omega)}\to +\infty$ as $n\to\infty$ a.e $t_0\in (0,T)$.
\end{proposition}

\begin{proof}
  Consider the problem \begin{equation}\begin{cases}
    u_t-\Delta_p u=0 & \text{in}\ (0,t_0) \times\R^N,\\
    u=\chi_\Omega  & \text{on}\ t=0 ,\\
  \end{cases}
\end{equation}
This problem has a unique strong solution $v\in W^{1,\infty}(0,t_0;L^1(\R^N))$;(see e.g. \cite{BeGa})
Moreover, there is a constant $K<1$ and a set $U\subseteq\Omega$ such that $v(x,t)\leq K$ for any $(x,t)\in U\times [0,t_0]$. On the other hand, by the maximum principle if we take $v_s(x,t):=v(x,t+s)$ we see that $v_s$ is a subsolution

\end{proof}

}

\subsection{Uniqueness of entropy large solutions.}
Here we want to prove the following
\begin{theorem}\label{unica}
The  entropy large solution for problem \rife{renpb} is unique.
\end{theorem}
The proof of this result is an easy consequence of the following contraction principle
\begin{theorem}\label{uniqpent}
  If $u,v$ are two entropy solutions corresponding to $u_0$, $v_0$ as initial data, if $(u_0-v_0)^+\in L^1(\Omega)$, then \begin{equation}
    \label{contpple} \|(u(t)-v(t))^+\|_{L^1(\Omega)}\leq \|(u_0-v_0)^+\|_{L^1(\Omega)}, \quad {\rm a.e\ } t\in [0,T)
  \end{equation}
\end{theorem}

\begin{proof}
  Suppose that $u=u(t,x)$ and $v=v(s,y)$ and take $l_1=v(s,y)$, $S=T_k^+$ as the constant and the truncation in Definition \ref{deflargeentropy} for $u$ and $l_2=u(t,x)$, $S=T_k^-$ for $v$. Then,

  $$\int_{Q} \eta |\nabla_x u(t,x))|^{p-2}\nabla_x u(x,t) \cdot \nabla_x T_k(T_h(u(t,x))-T_h(v(s,y)))^+\, dx\, dt$$$$+\int_{Q} T_k(T_h(u(t,x))-T_h(v(s,y)))^+ |\nabla_x u(t,x))|^{p-2}\nabla_x u(x,t)\cdot\nabla_x \eta \,dx \,dt$$$$=\int_{Q} j_{T_k^+,h,v(s,y)}(u(t,x))\eta_t\, dx\, dt$$
  for all $k,h>0$, $\eta\in \mathcal D(Q)$ and
   $$\int_{Q} \eta |\nabla_y v(s,y)|^{p-2}\nabla_y v(s,y)\cdot \nabla_y T_k(T_h(v(s,y))-T_h(u(t,x)))^-\, dy\, ds$$$$+\int_{Q} T_k(T_h(v(s,y))-T_h(u(t,x)))^- |\nabla_y v(s,y)|^{p-2}\nabla_y v(s,y)\cdot\nabla_y \eta \,dy \,ds$$$$=\int_{Q} j_{T_k^-,h,u(t,x)}(v(s,y))\eta_s\, dy\, ds$$ for all $k,h>0$, $\eta\in \mathcal D(Q)$.

   Let $0 \leq \phi \in
{\mathcal D}(0, \mathrm{T})$, $\psi\in \mathcal D(\Omega)$, $\rho_m$ a classical sequence of
mollifiers in $\Omega$ and $\tilde{\rho}_n$ a sequence of mollifiers in $\R$. Define
$$\eta_{m,n}(t, x, s, y):= \rho_m(x - y) \tilde{\rho}_n(t - s) \phi
\bigg(\frac{t+s}{2} \bigg)\psi\left(\frac{x+y}{2}\right).$$

Integrating the equations above in the other two variables and adding up both inequalities we get:
$$\int_{Q\times Q}\eta_{m,n}(|\nabla_x u(t,x)|^{p-2}\nabla_x u(x,t) \cdot\nabla_x T_k(T_h(u)-T_h(v))^+ $$
$$- |\nabla_y v(s,y)|^{p-2}\nabla_y v(s,y)\cdot\nabla_y T_k(T_h(u)-T_h(v))^+)$$$$+\int_{Q\times Q}(|\nabla_x u(t,x)|^{p-2}\nabla_x u(x,t) - |\nabla_y v(s,y)|^{p-2}\nabla_y v(s,y))\cdot(\nabla_x+\nabla_y)\eta_{m,n}T_k^+(T_h(u)-T_h(v))$$$$
+\int_{Q\times Q} |\nabla_y v(s,y)|^{p-2}\nabla_y v(s,y)\cdot\nabla_x \eta_{m,n}T_k(T_h(u)-T_h(v))^+$$
$$-\int_{Q\times Q}|\nabla_x u(t,x))|^{p-2}\nabla_x u(x,t) \cdot\nabla_y \eta_{m,n}T_k(T_h(u)-T_h(v))^+$$$$= \int_{Q\times Q}j_{T_k^+,h,v(s,y)}(u)(\eta_{m,n})_t+j_{T_k^-,h,u(x,t)}(v)(\eta_{m,n})_s)\,.$$

Now, using Green's Formula, we obtain
$$\int_{Q\times Q}\eta_{m,n}( |\nabla_x u|^{p-2}\nabla_x u\cdot\nabla_x T_k(T_h(u)-T_h(v))^+- |\nabla_y v|^{p-2}\nabla_y v\cdot\nabla_y T_k(T_h(u)-T_h(v))^+)$$$$+\int_{Q\times Q} |\nabla_y v|^{p-2}\nabla_y v\cdot\nabla_x\eta_{n,m} T_k(T_h(u)-T_h(v))^+-\int_{Q\times Q} |\nabla_x u|^{p-2}\nabla_x u\cdot\nabla_y \eta_{n,m} T_k(T_h(u)- T_h(v))^+$$$$=\int_{Q\times Q}\eta_{m,n}( |\nabla_x u|^{p-2}\nabla_x u- |\nabla_y v|^{p-2}\nabla_y v)\cdot (\nabla_x T_k(T_h(u)-T_h(v))^++\nabla_y T_k(T_h(u)- T_h(v))^+\geq 0\,.$$

where the last inequality is due to the monotonicity of the $p-$laplace operator.

Therefore,
$$\int_{Q\times Q}(|\nabla_x v|^{p-2}\nabla_x v-|\nabla_y u|^{p-2}\nabla_y u)\cdot(\nabla_x+\nabla_y)\eta_{m,n}T_k^+(T_h(u)-T_h(v))$$$$\leq \int_{Q\times Q}j_{T_k^+,h,v}(u)(\eta_{m,n})_t+j_{T_k^-,h,u)}(v)(\eta_{m,n})_s)\,.$$
By dividing the above expression by $k$ and letting $k\to\infty$,  we get
$$-\int_{Q\times Q}(u-v)sign_0^+(T_h(u)-T_h(v))((\eta_{m,n})_t+(\eta_{m,n})_s)
 $$$$+\int_{Q\times Q}(|\nabla_x v|^{p-2}\nabla_x v-|\nabla_y u|^{p-2}\nabla_y u)\cdot\nabla\eta_{m,n}sign_0^+(T_h(u)-T_h(v))\leq 0\,.$$

Passing to the limit when $n,m\to\infty$ yields:
$$-\int_{Q}(u-v)sign_0^+(T_h(u)-T_h(v))\phi'(t)\psi(x)
 $$$$+\int_{Q}\phi(t)(|\nabla v|^{p-2}\nabla v -|\nabla u|^{p-2}\nabla u)\cdot\nabla\psi(x)sign_0^+(T_h(u)-T_h(v))\leq 0\,.$$Having in mind the lower bound \eqref{lowerbd}, we can find $\varepsilon(h)>0$ such that $T_h(u)=T_h(v)$ in $\Omega\setminus\Omega_{\varepsilon(h)}\times [0,T)$, where
$$
 \Omega_{\varepsilon(h)}:=\{x\in\Omega: \text{\rm dist}(x,\partial\Omega)\geq\varepsilon(h)\}\,.
 $$\bk
 Then,
 $$\int_{Q}\phi(t)(|\nabla v|^{p-2}\nabla v -|\nabla u|^{p-2}\nabla u)\cdot\nabla\psi(x)sign_0^+(T_h(u)-T_h(v))= 0$$ for all $\psi\in \mathcal D(\Omega)$ such that $supp(\nabla\psi)\subseteq \Omega\setminus\Omega_{\varepsilon(h)}$. Therefore
$$-\int_{Q}(u-v)sign_0^+(T_h(u)-T_h(v))\psi(x)\phi'(t)\leq 0$$for all $\psi\in \mathcal D(\Omega)$ such that $supp(\nabla\psi)\subseteq \Omega\setminus\Omega_{\varepsilon(h)}$. Therefore,
 $$\frac{d}{dt}\int_\Omega (u(t)-v(t))sign_0^+(T_h(u(t))-T_h(v(t)))\psi(x)\,dx\leq 0\,,\quad {\rm a.e. \ } t\in [0,T]\,,$$for all $\psi\in \mathcal D(\Omega)$ such that $supp(\nabla\psi)\subseteq \Omega\setminus\Omega_{\varepsilon(h)}$. Thus, a.e. $t\in [0,T]$ $$\int_\Omega (u(t)-v(t))sign_0^+(T_h(u(t))-T_h(v(t)))\psi(x)\,dx$$$$\leq \int_\Omega (u(0)-v(0))sign_0^+(T_h(u(0))-T_h(v(0)))\psi(x)\,dx$$for all $\psi\in \mathcal D(\Omega)$ such that $supp(\nabla\psi)\subseteq \Omega\setminus\Omega_{\varepsilon(h)}$
 in particular for $\psi_h$ begin a cut-off function such that $\psi_h=1$ in $\Omega_{\frac{\varepsilon(h)}{2}}$. Letting $h\to +\infty$, and applying Fatou's Lemma we finally obtain $$\int_\Omega (u(t)-v(t))^+\,dx\leq \liminf_{h\to\infty} \int_\Omega (u(0)-v(0))sign_0^+(T_h(u(0))-T_h(v(0)))\psi_h(x)\,dx$$$$=\int_\Omega (u(0)-v(0))^+\,dx $$ a.e. $t\in [0,T)$
\end{proof}

\begin{proof}[Proof of Theorem \ref{maint}]
The proof follows by gathering together Theorem \ref{esiste} and Theorem \ref{unica}.
\end{proof}

\section{Further remarks and generalizations}
\subsection{Why $p<2$ ?} In this section we want to justify the assumption $p<2$. In the case $p\geq 2$ not only the possibility to find local estimates is forbidden but furthermore a nonexistence result  for large solutions can be proved in the sense we will specify in a while. Roughly speaking, we show that, in this case, the approximating solutions turn  to explode on a set of positive measure yielding, as a by-product, the impossibility to construct blow-up solutions.
Let us consider the approximating problems

\begin{equation}\label{pm2}
\begin{cases}
    (u_n)_{t}-\Delta_p u_n =0 & \text{in}\ Q,\\
    u_n=0  & \text{on}\ \{0\} \times \Omega,\\
 u_n=n &\text{on}\ (0,T)\times\partial\Omega\,,
  \end{cases}
\end{equation}
whose weak solutions are weak subsolutions of problems (\ref{senzah}).

\medskip
In the particular case $p=2$ we can prove something more precise, namely we have the following
\begin{proposition}
Let $p=2$ and let $u_n$ be the weak solution for problem \rife{senzah}. Then $u_n\to \infty$ as $n$ diverges at any point $(t,x)\in Q$.
\end{proposition}
\begin{proof}
Consider the solution to the  auxiliary problem
$$
\begin{cases}
    v_{t}-\Delta v  =0 & \text{in}\ Q,\\
    v=1  & \text{on}\ \{0\} \times \Omega,\\
 v=0 &\text{on}\ (0,T)\times\partial\Omega\,.
  \end{cases}
$$
It is easy to see that the function $v_n =n(1-v)$ is the unique weak solution to problem \rife{pm2}. Moreover, observe that, by strong maximum principle, $v(t,x)<1$ for any $(t,x)\in Q$. So that, letting $n$ going to infinity, we get that $v_n\to\infty$, and consequently $u_n\to\infty$ as $n$ diverges at any point $(t,x)\in Q$.
\end{proof}

If $p>2$, due to the finite speed of propagation feature, the previous argument is no longer available and something different can happen. Anyway, nonexistence of large solutions in the sense specified above can still be proved as a consequence of the following

\begin{proposition}\label{nonexistence}
Let $p>2$ and let $u_n$ be the weak solution for problem \rife{pm2}. Then there exist a set $E\subseteq Q$ of positive measure such that  $u_n\to \infty$ as $n$ diverges at any point $(t,x)\in E$.
\end{proposition}
\begin{proof}
Consider the solution to the  auxiliary problem
$$
\begin{cases}
    v_{t}-\Delta_p v  =0 & \text{in}\ Q,\\
    v=1  & \text{on}\ \{0\} \times \Omega,\\
 v=0 &\text{on}\ (0,T)\times\partial\Omega\,.
  \end{cases}
$$
 Again,  consider $v_n =n(1-v)$.
First of all, due to the boundary condition and the regularity of the solution, there exists a set $E$ of positive measure on $Q$ such that $v(x,t)<1 $ on $E$. Our aim is to show that $v_n$ is a subsolution to problem \rife{senzah}, this fact will easily imply the result.

Our first claim is to check that, with these data, $\Delta_p v \leq 0$ in $\mathcal{D}'(Q)$, or, equivalently, that
$$
- \int_Q v\eta_t \leq  0
$$
for any $0\leq \eta\in \mathcal{D}(Q)$. This is a consequence of comparison principle that easily implies $ v(t+s,x)\leq v(t,x)$ a.e. on $Q$.  Therefore, for fixed $\eta$ and $\tau>0$ small enough we have
$$
\int_0^T \into \frac{\eta (t+\tau)-\eta(t)}{\tau} v=\int_{\tau}^{T}\into\frac{\eta(t)}{\tau} (v(t)-v(t-\tau))\geq 0.
$$
that implies the result by passing to the limit on $\tau$.

Now, since $\Delta_p v_n \geq 0$, we can check that $v_n$ is a subsolution for problem \rife{senzah}, in fact, recalling that $p>2$, we have
$$
(v_n)_t =-n \Delta_p v = \frac{1}{n^{p-2}}\Delta_p v_n\leq  \Delta_p v_n
$$
in $ \mathcal{D}'(Q)$, and this concludes the proof.
\end{proof}

\bk
\subsection{General monotone operators}
The existence result of Theorem \ref{esiste} can be straightforwardly extended to more general monotone operators of Leray-Lions type.
Consider the problem
\begin{equation}\label{nonlinu}
\begin{cases}
    u_t-\dive(a(t,x,\nabla u))=0 & \text{in}\ Q,\\
u=u_0  & \text{on}\ \{0\} \times \Omega,\\
u=+\infty &\text{on}\ (0,T)\times\partial\Omega,
  \end{cases}
\end{equation}
where  $a : Q \times \rn \to \rn$ is a Carath\'eodory function (i.e., $a(\cdot,\cdot,\xi)$
is measurable on $Q$ for every $\xi$ in $\rn$, and $a(t,x,\cdot)$ is
continuous on $\rn$ for almost every $(t,x)$ in $Q$), such that the
following holds:
\be
a(t,x,\xi) \cdot \xi \geq \al|\xi|^p,
\label{coercp}
\ee
\be
|a(t,x,\xi)| \leq \beta[b(t,x) + |\xi|^{p-1}],
\label{cont}
\ee
\be
[a(t,x,\xi) - a(t,x,\eta)] \cdot (\xi - \eta) > 0,
\label{monot}
\ee
for almost every $(t, x)$ in $Q$, for every $\xi, \eta \in \rn$, with
$\xi \neq \eta$, where $1<p<2$,
$\al$ and $\beta$ are two positive constants, and
$b$ is a nonnegative function in $L^{p'}(Q)$. Here again $u_0\in L^1_{loc}(\Omega)$.

{\begin{definition}\label{deflargeentropygen}A measurable function $u:Q\to \R$ is an entropy large solution of \eqref{nonlinu}  if $u\in C(0,T;L^1_{loc}(\Omega))$, $T_k(u)\in L^p(0,T;W^{1,p}_k(\Omega))$, for all $k>0$, $|\nabla u|^{p-1}\in L^1_{loc}(Q)$,  and $$\int_{Q} \eta a(t,x,\nabla u)\cdot \nabla S(T_h(u)-T_h(l))+\int_{Q} S(T_h(u)-T_h(l))a(t,x,\nabla u)\cdot\nabla \eta $$$$=\int_{Q} j_{S,h,l}(u)\eta_t$$
for all $S\in\mathcal P$, $h>0$ $l\in\R$, $\eta\in \mathcal D(Q)$ and $j_{S,h,l}(r):=\int_l^r S(T_h(s)-T_h(l))\, ds$. Moreover, $u(0)=u_0$ in $L^1_{loc}(\Omega)$.
\end{definition}}
We have
\begin{theorem}\label{esisteg}
There exists an entropy large solution for problem \rife{nonlinu}.	
\end{theorem}

The proof runs exactly as in Theorem \ref{esiste} keeping in mind that both the estimates in \cite{lepe} and the the local strong convergence of truncations (as in \cite{lp}) hold true for a such general framework.

In contrast, we saw that uniqueness is based on the estimate
\begin{equation}
   \label{lowerbdg}\frac{C_0t^\frac{1}{2-p}}{{\rm dist}(\partial\Omega)^\frac{p}{2-p}}\leq u(x,t)
 \end{equation}
which is no longer available here. On the other hand, note that, in the proof of uniqueness,  we only need a lower bound providing that $u(x,t)\to\infty$ uniformly as $dist(x,\partial\Omega)\to 0$. Therefore we can state the following partial uniqueness result.
\begin{theorem}\label{unig}
There is at most one entropy large solution for problem \rife{nonlinu} satisfyng $u(x,t)\to\infty$ uniformly  as $dist(x,\partial\Omega)\to 0$.	\end{theorem}
The proof of this fact is an easy extension of the argument in the proof of Theorem \ref{unica}.

\section{The case $p=1$. The total variation flow}

In this section we study the large entropy solutions for the Total Variation Flow; i.e.

\begin{equation}
  \label{large1lp}\left\{\begin{array}
    {cc}\displaystyle \frac{\partial u}{\partial t}=\div\left(\frac{Du}{|Du|}\right) & \quad {\rm in \ } (0,T)\times\Omega \\ \\ u(t,x)=+\infty & \quad {\rm on \ } (0,T)\times\partial\Omega \\ \\ u(0,x)=u_0 & \quad {\rm in \ } \Omega
  \end{array}\right.\,,
\end{equation}
with $u_0\in L^1(\Omega)$.

To prove existence of solutions of problem \eqref{large1lp} we use the techniques of completely accretive
operators and the Crandall-Liggett semigroup generation theorem in \cite{crli}. {We point out that, contrary to the case $1<p<2$ our framework will be $L^1(\Omega)$ and not $L^1_{loc}(\Omega)$. This is due to the use of the nonlinear semigroup theory but, as we show in this case, if the initial data is in $L^1(\Omega)$, then the solution is still in $L^1(\Omega)$. This fact is not true anymore if $p>1$ by the estimate \eqref{lowerbd}}.

\medskip
We recall now the notion of $m-$completely accretive operators introduced in \cite{becr}.
Let $\mathcal M(\Omega)$ be the space of measurable functions in $\Omega$. Given $u, v\in \mathcal M(\Omega)$,
we shall write
$$u<<v {\rm \ if \ and \ only \ if \ } \int_\Omega j(u) dx\leq \int_\Omega
j(v) dx$$
for all $j\in J_0$ where
$J_0:=\{ j: \R\to [0, +\infty],$ convex, l.s.c., $j(0)=0\}$. Let $B\subseteq {\mathcal M}(\Omega)\times {\mathcal M}(\Omega)$ be an operator.  $B$ is said to be completely accretive if
$$u-\overline u <<u-\overline u +\lambda(v-\overline v) {\rm \ for \ all\ } \lambda>0 {\rm \ and \ all \ } (u, v), (\overline u, \overline v ) \in B. $$
Let
$P_0:=\{p\in C^\infty(\R) : 0\leq p'\leq 1,$ $supp( p')$ is compact and $0\notin supp(p)\}$.
If $B\subseteq L^1(\Omega)\times L^1 (\Omega)$, then $B$ is completely accretive if and only if
$$
\int_\Omega p(u-\overline u )(v-\overline v)\geq 0 {\rm \ for \ any  \ } p\in P_0,\quad  (u, v), (\overline u , \overline v ) \in B.$$
A completely accretive operator $B$ in $L^1 (\Omega)$ is said to be m-completely
accretive if $R(I+\lambda B)=L^1 (\Omega)$ for any $\lambda>0$.

We consider now the elliptic problem associated to \eqref{large1lp}:

\begin{equation}
  \label{elarge1lp}\left\{\begin{array}
    {cc}\displaystyle u-v=\div\left(\frac{Du}{|Du|}\right) & \quad {\rm in \ } \Omega \\ \\ u(t,x)=+\infty & \quad {\rm on \ } \partial\Omega \\
  \end{array}\right.\,,
\end{equation}
with $v\in L^1(\Omega)$. We proceed as in the case that $1<p<2$; i.e. approximating the problem \eqref{elarge1lp} by problems
\begin{equation}
  \label{elarge1lpn}\left\{\begin{array}
    {cc}\displaystyle u-v=\div\left(\frac{Du}{|Du|}\right) & \quad {\rm in \ } \Omega \\ \\ u(t,x)=n& \quad {\rm on \ } \partial\Omega \\
  \end{array}\right.\,
\end{equation}and letting $n\to\infty$.

In \cite{ACM} it is defined the following operator associated to problem \eqref{elarge1lpn}:

\begin{definition}
  \label{depopn} $(u,v)\in \mathcal A_n$ iff $u,v\in L^1(\Omega)$, $u\in TBV(\Omega)$ and there exists $\z\in X(\Omega)$ with $\|\z\|_{\lio}\leq 1$, $v=-\div \z $ in $\mathcal D'(\Omega)$ such that $$[\z,\nu]\in {\rm sign} (n-u)\,,\qquad \mathcal H^{N-1}-{\rm a.e \ in \ } \partial\Omega$$ and for all $S\in \mathcal P,$
    $$\int_\Omega (\z,D S(u))=\int_\Omega |DS(u)|\,.$$
\end{definition}

and the following result is proved:
\begin{theorem}\label{thAn}
    The operator $\mathcal A_n$ is m-completely accretive in $L^1(\Omega)$ with dense domain.
  \end{theorem}

Letting now formally $n\to\infty$ in Definition \ref{depopn} we define the following operator in $L^1(\Omega)$ as the one corresponding to problem (\ref{elarge1lp}):

  \begin{definition}\label{52}
    We say that $(u,v)\in \mathcal A$ iff $u,v\in L^1(\Omega)$, $u\in TBV(\Omega)$ and there exists $\z\in X(\Omega)$ with $\|\z\|_{\lio}\leq 1$, $v=-\div \z $ in $\mathcal D'(\Omega)$ such that $$[\z,\nu]=1\,,\qquad \mathcal H^{N-1}-{\rm a.e \ in \ } \partial\Omega$$ and
    $$\int_\Omega (\z,D S(u))=\int_\Omega |DS(u)|\,,\quad {\rm for \ all \ } S\in \mathcal P.$$
  \end{definition}

We also have in this case,
\begin{theorem}\label{thA}
    The operator $\mathcal A$ is m-completely accretive in $L^1(\Omega)$ with dense domain.
  \end{theorem}

  To prove this result we need some auxiliary results: a uniform bound on solutions given by Remark  \ref{rem} and a characterization of the operator given in Proposition \ref{prop} below.

\begin{definition}
  Given $v\in L^\infty(\Omega)$, we say that $u\in TBV(\Omega)$ is a {distributional} solution  of $v-u=-\div(\frac{Du}{|Du|})$ if there is $\z\in X(\Omega)$ with $\|\z\|_{\lio}\leq 1$ such that $$v-u=-\div \z \,\quad \mbox{in \ } \mathcal D'(\Omega)$$ and $$\int_\Omega (\z,DS(u))=\int_\Omega |DS(u)|\,,\quad {\rm for \ all\ } S\in \mathcal P.$$
  \end{definition}
\begin{lemma}\label{complemma}
  Given $v\in L^\infty(\Omega)$, then any distributional solution $u\in TBV(\Omega)$ of $v-u=-\div(\frac{Du}{|Du|})$ verifies the following $L^\infty$ bound for any ball $B_s(x_0)\subseteq\Omega$ $$\|u\|_{L^\infty(B_s(x_0))}\leq \|v\|_{\lio}+\frac{Per(B_s(x_0))}{|B_s(x_0)|}$$
\end{lemma}

\begin{proof}
  First of all we note that considering $\overline u(x):=\|v\|_{\lio}+\frac{Per(B_s(x_0))}{|B_s(x_0)|}$, $\overline \z(x):=\frac{x-x_0}{s}$, we have: \begin{equation}\label{compball}\|v\|_{\lio}-\overline u=-\div \overline \z\end{equation} in $\mathcal D'(B_s(x_0))$, and $[\overline \z, \nu^{B_s}]=1$. We multiply \rife{compball}  by $(T_k(u)-T_k(\overline u))^+$ and integrate in $B_s(x_0)$,  to obtain $$\int_{B_s(x_0)}(T_k(u)-T_k(\overline u))^+(\|v\|_{\lio}-\overline u)=\int_{B_s(x_0)}(\overline \z,D(T_k(u)-T_k(\overline u))^+)$$$$-\int_{\partial(B_s(x_0))}(T_k(u)-T_k(\overline u))^+\, d\mathcal H^{N-1}\,.$$
  On the other hand, since $v-u=-\div(\z)$ in $\mathcal D'(\Omega)$, we multiply this equation by $-(T_k(u)-T_k(\overline u))^+$, and integrating in $B_s(x_0)$ we get $$-\int_{B_s(x_0)}(T_k(u)-T_k(\overline u))^+(v-u)=-\int_{B_s(x_0)}(\z,D(T_k(u)-T_k(\overline u))^+)$$$$+\int_{\partial(B_s(x_0))}[\z,\nu^{B_s}](T_k(u)-T_k(\overline u))^+\, d\mathcal H^{N-1}.$$ Adding both equalities we get, $$\int_{B_s(x_0)} (T_k(u)-T_k(\overline u))^+(\|v\|_{\lio}-v+(u-\overline u))$$$$=-\int_{B_s(x_0)}(\z-\overline \z,D(T_k(u)-T_k(\overline u))^+)-\int_{\partial(B_s(x_0))}(1-[\z,\nu^{B_s}])(T_k(u)-T_k(\overline u))^+\,d\mathcal H^{N-1}\leq 0\,,$$
  where in the last inequality we used \rife{acota}. Therefore, letting $k\to\infty$ we obtain the desired result.
  \end{proof}

  \begin{remark} \label{rem}Since $\Omega$ satisfies a uniform interior ball condition, by Lemma \ref{complemma} we have that given $v\in L^\infty(\Omega)$ any distributional solution $u\in TBV(\Omega)$ of $v-u=\div(\frac{Du}{|Du|})$ is uniformly bounded; i.e. it satisfies $$\|u\|_{\lio}\leq \|v\|_{\lio}+\frac{N}{s_0}\,,$$where $s_0$ is given by the uniform interior ball condition.\end{remark}

The following characterization of $\mathcal A$ is essential to prove Theorem \ref{thA} and its proof reduces to the one of Proposition 2 in \cite{ACM}:

  \begin{proposition}\label{prop}
    $(u,v)\in \mathcal A$ iff $u,v\in L^1(\Omega)$, $u\in TBV^+(\Omega)$ and there exists $\z\in X(\Omega)$ with $\|\z\|_{\lio}\leq 1$, $v=-\div \z $ in $\mathcal D'(\Omega)$ such that
    \begin{equation}
      \label{1ellipticdef}\int_\Omega (w-S(u))v\leq \int_\Omega (\z,Dw)-\int_\Omega |D S(u)|-\int_{\partial\Omega}(w-S(u))\,d\mathcal H^{N-1}
    \end{equation}for all $w\in BV(\Omega)\cap L^\infty(\Omega)$ and $S\in\mathcal P$.
  \end{proposition}


  \begin{proposition}\label{auxprop}
    $L^\infty(\Omega)\subset R(I+\mathcal A)$ and $D(\mathcal A)$ is dense in $L^1(\Omega)$
  \end{proposition}
  \begin{proof}
    Let $v\in L^\infty(\Omega)$ and take $n>>M$ with $M:=\|v\|_{\lio}+\frac{N}{s_0}$ with $s_0$ goven by the uniform interior ball condition. Then, by Theorem \ref{thAn}, there are $u_n\in TBV(\Omega)$ and $\z_n\in X(\Omega)$ such that $v-u_n=-\div (\z_n)$,
   in $\mathcal D'(\Omega)$. Moreover,  $$[\z_n,\nu]\in {\rm sign} (n-u_n)\,,\qquad \mathcal H^{N-1}-{\rm a.e \ in \ } \partial\Omega$$ and for all $S\in \mathcal P,$
    $$\int_\Omega (\z_n,D S(u_n))=\int_\Omega |DS(u_n)|\,.$$

    By Remark \ref{rem}, we have that $\|u_n\|_{\lio}\leq M$. Thus, $$[\z_n,\nu]=1\,,\qquad \mathcal H^{N-1}-{\rm a.e \ in \ } \partial\Omega$$ which implies that $(u_n,v-u_n)\in \mathcal A$.

    \medskip
    In order to prove the density of the domain, we show that $\mathcal C_0^\infty(\Omega)\subseteq \overline{D(\mathcal A)}^{L^1(\Omega)}$. Let $v\in \mathcal C_0^\infty(\Omega)$. Then, $v\in R(I+\frac{1}{n}\mathcal A)$ for all $n\in \N$. Then, there exist $u_n\in D(\mathcal A)$, $\z_n\in X(\Omega)$ with $\|\z_n\|\leq 1$ such that $n(v-u_n)=-\div (\z_n)$ in $\mathcal D'(\Omega)$ such that
    $$n\int_\Omega (w-u_n)(v-u_n)\leq \int_\Omega (\z_n,Dw)-\int_\Omega |D u_n|-\int_{\partial\Omega}(w-u_n)\,d\mathcal H^{N-1}\,,$$ for all $w\in BV(\Omega)\cap L^\infty(\Omega)$. Taking $w=v$ in the previous inequality we get $$\int_\Omega (v-u_n)^2\leq \frac{1}{n}\int_\Omega |\nabla v|.$$ Then, $u_n\to v$ in $L^2(\Omega)$ and thus $v\in \overline{D(\mathcal A)}^{L^1(\Omega)}$.
  \end{proof}

\noindent{\it Proof of Theorem \ref{thA}:} Let $(u,v)$, $(\overline u,\overline v) \in \mathcal A$ and $p\in \mathcal P_0$  and consider $p_\infty:=\lim_{r\to +\infty}p(r)$. Let $\z,\overline \z\in X(\Omega)$ such that $\|\z\|_{\lio},\|\overline \z\|_{\lio}\leq 1$, $v=-\div \z$, $\overline v=-\div(\overline \z)$ in $\mathcal D'(\Omega)$ and such that for any $w\in BV(\Omega)\cap L^\infty(\Omega)$ and any $S\in\mathcal P$,
\begin{equation}\label{pfineq1}\int_\Omega (w-S(u))v\leq \int_\Omega (\z,Dw)-\int_\Omega |DS(u)|-\int_{\partial\Omega}(w-S(u))\,d\mathcal H^{N-1},\end{equation}
\begin{equation}\label{pfineq2}\int_\Omega (w-S(\overline u))\overline v\leq \int_\Omega (\overline \z,Dw)-\int_\Omega |DS(\overline u)|-\int_{\partial\Omega}(w-\overline S(u))\,d\mathcal H^{N-1}\,.\end{equation}Taking $w=S(u)-p(T_k(u)-T_k(\overline u))$ as test function in (\ref{pfineq1}),  $w=S(\overline u)+p(T_k(u)-T_k(\overline u))$ in (\ref{pfineq2}) and adding both inequalities we get, $$\int_\Omega p(T_k(u)-T_k(\overline u))(v-\overline v)\geq \int_\Omega (\z-\overline \z, Dp(T_k(u)-T_k(\overline u)))\geq 0\,.$$ Letting $k\to\infty$, we get that $\mathcal A$ is completely accretive.

Thanks to Proposition \ref{auxprop}, we only have to prove that $\mathcal A$ is closed. Let $(u_n,v_n)\in \mathcal A$ such that $(u_n,v_n)\to (u,v)$ in $L^1(\Omega)\times L^1(\Omega)$. Then, there exist $\z_n\in X(\Omega)$ such that $\|\z_n\|_{\lio}\leq 1$, $v_n=-\div(\z_n)$ in $\mathcal D'(\Omega)$ and $$\int_\Omega (w-S(u_n))v_n\leq \int_\Omega (\z_n,Dw)-\int_\Omega |DS(u_n)|-\int_{\partial\Omega}(w-S(u_n))\,d\mathcal H^{N-1}$$ for every $w\in BV(\Omega)\cap L^\infty(\Omega)$ and $S\in\mathcal P$. Since $\|\z_n\|_{\lio}\leq 1$, we find $\z\in X(\Omega)$ with $\|\z\|_{\lio}\leq 1$ such that $\z_n\to \z$ weakly$^*$ in $L^\infty(\Omega;\R^N)$. Moreover, since $v_n \to v$ in $\luo$,  we have  $v=-\div(\z)$. Finally, having in mind the lower semicontinuity of the functional given by $$u\mapsto \int_\Omega |DS(u)|+\int_{\partial\Omega}|S_\infty-S(u)|\,d\mathcal H^{N-1}$$ where $S_\infty:=\lim_{r\to+\infty}S(r)$ and taking limits we get $$\int_\Omega (w-S(u))v\leq \int_\Omega (\z,Dw)-\int_\Omega |DS(u)|-\int_{\partial\Omega}(w-S(u))\,d\mathcal H^{N-1}\,,\hfill $$
that concludes the proof. \hfill$\qed$

\subsection{The semigroup solution}

By Theorem \ref{thA}, according to Crandall-Ligget's Generation Theorem, for every initial datum $u_0\in L^1(\Omega)$, there exists a unique mild solution $u\in C(0,T;L^1(\Omega))$ of the evolution problem \begin{equation}\label{semipr}\frac{d u}{d t}+\mathcal A\ni 0\,,\quad u(0)=u_0\end{equation} given by $u(t)=\mathcal{S}(t)u_0$, where $(\mathcal{S}(t))_{t\geq 0}$ is the contraction semigroup in $\luo$ generated by $\mathcal{A}$ which is given by the exponential formula $$S(t)u_0=\lim_{n\to\infty}(I+\frac{t}{n}\mathcal A)^{-n}u_0.$$

Moreover, since $\mathcal A$ is completely accretive, if the initial datum $u_0\in D(\mathcal A)\cap \lio$, then the mild solution is a strong solution; i.e.

\begin{lemma}\label{mildlemma}
  Let $u_0\in D(\mathcal A)\cap L^\infty(\Omega)$, and let $u(t)=\mathcal{S}(t)u_0$ be the mild solution of (\ref{semipr}). Then, $u\in L^1(0,T;BV(\Omega)\cap L^\infty(\Omega))\cap W^{1,1}(0,T;L^1(\Omega))$ for every $T>0$ and there exists  $\z(t)\in X(\Omega)$ a.e. $t\in [0,T]$ verifying \begin{equation}
    \label{linftypr}\int_\Omega u'(t)(S(u(t))-w)\leq \int_\Omega (\z(t),Dw)-\int_\Omega |DS(u(t))|-\int_{\partial\Omega}(w-S(u))\,d\mathcal H^{N-1} \end{equation}for every $w\in BV(\Omega)\cap L^\infty(\Omega)$, $S\in\mathcal P$ and a.e. $t\in [0,T]$. Moreover, $u(t)$ is characterized as follows: there exists $\z(t)\in X(\Omega)$, $\|\z\|_{\lio}\leq 1$, $u'(t)=\div \z(t)$ in $D'(\Omega)$ a.e $t\in[0,T]$, and $$\int_\Omega (\z(t),DS(u(t)))=\int_\Omega |DS(u(t))|\,,\quad \forall S\in\mathcal P\,,$$$$[\z(t),\nu]=1\,,\qquad \mathcal H^{N-1}-{\rm a.e. \ on \ }\partial\Omega \ \ {\rm and \ a.e. \ }t\in [0,T]$$

\end{lemma}

\begin{proof}The result is a consequence of the nonlinear semigroup theory and it remains only to prove that $u(t)\in L^\infty(\Omega)$ for a.e. $t\in [0,T]$. Let us give the  construction of $u$ in terms of Crandall-Ligget's Theorem.

  Let $u(t)=S(t)u_0$, then
  the set $K$ consisting of the values of $t\in [0,T]$ for which either $u$ is not differentiable at $t$ or $t$ is not a Lebesgue point  for $u'$ or $u'(t)+\mathcal A\not\ni 0$  is null. Then, since $u'\in L^1(0,T;L^1(\Omega))$, for any $\varepsilon>0$, there exists a partition $0=t_0<t_1<\ldots<t_{n-1}\leq T<t_n$ such that $t_k\notin K$, $t_k-t_{k-1}<\varepsilon$ for $k=1,\ldots,n$ and $$\sum_{k=1}^n\int_{t_{k-1}}^{t_k}\|u'(s)-u'(t_k)\|_{\luo}\,ds<\varepsilon\,.$$
  We define $$u_\varepsilon(t):=u_0\chi_{[t_0,t_1]}(t)+\sum_{k=1}^n u(t_k)\chi_{]t_{k-1},t_k]}(t)\,.$$ Then $u_\varepsilon\to u$ in $C(0,T;L^1(\Omega))$. Moreover, by Remark \ref{rem} $$\|u_\varepsilon\|_{\lio}\leq \|u_0\|_{\lio}+\frac{TN}{s_0}$$ which implies \begin{equation}\label{boundtvf}\|u\|_{\lio}\leq M_T:=\|u_0\|_{\lio}+\frac{TN}{s_0}\,.\end{equation}

\end{proof}
\subsection{Existence and uniqueness for $L^1$-initial data}

We are now in the position to prove the existence and uniqueness of an entropy large solution.

Here is the definition of solution  that is the natural adaptation of Definition \ref{deflargeentropy} to the case $p=1$.
\begin{definition}\label{Defentropy}
  A measurable function $u$ is  an entropy large solution of (\ref{large1lp}) in $Q=(0,T)\times \Omega$ if $u\in C(0,T;L^1(\Omega))$, $p(u(\cdot))\in L^1_w(0,T;BV(\Omega))$, $\forall p\in \mathcal P$ and there exist $(\z(t),\xi(t))\in Z(\Omega)$ with $\|\z\|_{\lio}\leq 1$ and $\xi\in (L^1(0,T;BV(\Omega)_2))^*$ such that $\xi$ is the time derivative of $u$ in $(L^1(0,T;BV(\Omega)_2))^*$,
  $\xi=\div \z$ in $(L^1(0,T;BV(\Omega)))^*$, \begin{equation}\label{boundary}[\z(t),\nu]=1\,\quad \mathcal{H}^{N-1}{\rm a.e. \ on \ }\partial\Omega \, \, {\rm and \ a.e \ } t\in (0,T)\end{equation} satisfying
  \begin{equation}
    \label{entropylargetvf}-\int_{Q} j_{S,h,l}(u)\eta_t+\int_{Q}\eta |DS(T_h(u)-T_h(l))|+\int_{Q} S(T_h(u)-T_h(l))\z\cdot\nabla\eta \leq 0\,,
  \end{equation}for all $l\in \R$, $h>0$, $0\leq \eta\in \mathcal D(Q)$, $S\in \mathcal P$ and $j_{S,h,l}(r):=\int_l^r S(T_h(s)-T_h(l))\, ds$.
\end{definition}

\medskip

The main result of this section is the following

\begin{theorem}\label{theoremp1}
  Given $u_0\in L^1(\Omega)$, there exists a unique entropy large solution of (\ref{large1lp}) in $Q$ for all $T>0$. Moreover, if $u, v$ are the entropy large solutions corresponding to the initial data $u_0, v_0$ respectively, then $$\|u(t)-v(t)\|_{\luo}\leq \|u_0-v_0\|_{\luo}$$ for all $t\geq 0$. {Furthermore, if $u_0\in L^\infty(\Omega)$ then \begin{equation}\label{bound}u(t)\leq \|u_0\|_{\lio}+\frac{TN}{s_0}\, \, {\rm  \ a.e \ } t\in (0,T)\,,\end{equation}}
  where $s_0$ is given by the uniform interior ball condition.
\end{theorem}

\begin{proof} The proof of existence is an easier version of the proof of Theorem 1 in \cite{ACM}. { For the convenience of the reader, we give the main steps.

\medskip
Let $u_0\in L^1(\Omega)$ and $\{\mathcal{S}(t)\}_{t\geq 0}$ be the contraction semigroup in $L^1(\Omega)$ generated by $\mathcal A$. We prove now that the mild solution $u(t)=\mathcal{S}(t)u_0$ is also an entropy solution of (\ref{large1lp}). The proof is divided into several steps.

\medskip

\noindent{\it Step 1. Approximating Sequences.}  Since $D(\mathcal A)\cap L^\infty(\Omega)$ is dense in $L^1(\Omega)$, given $u_0\in L^1(\Omega)$, there is a sequence $u_{0,n}\in D(\mathcal A)\cap L^\infty(\Omega)$ such that $u_{0,n}\to u_0\in L^1(\Omega)$. Denoting $u_n(t):=\mathcal{S}(t)u_{0,n}$, we have that $u_n\to u\in C([0,T];L^1(\Omega))$ for all $T>0$. By Lemma \ref{mildlemma}, $u_n\in L^1(0,T;BV(\Omega)\cap L^\infty(\Omega))\cap W^{1,1}(0,T;L^1(\Omega))$ for all $T>0$ and there exists $\z_n(t)\in X(\Omega)$, $\|\z_n(t)\|_{\lio}\leq 1$, $u_n'(t)=\div(\z_n(t))$ in $\mathcal D'(\Omega)$ a.e. $t\in [0,+\infty[$ such that \begin{equation}
  \label{anzeq}\int_\Omega(\z_n(t),DS(u_n(t)))=\int_\Omega |DS(u_n(t))|\,,\quad \forall S\in \mathcal P
\end{equation}
\begin{equation}
  \label{boundaryt}[\z_n(t),\nu]=1\,\quad \mathcal H^{N-1}{\rm -a.e. \ on \ }\partial\Omega
\end{equation}

From here, \eqref{bound} is a direct consequence of (\ref{boundtvf}).
\medskip

\noindent{\it Step 2. Convergence of the sequences and identification of the limit.} Working exactly as in Steps 2, 4 and 5 of Theorem 1 and Lemma 2 in  \cite{ACM}, we prove that \begin{itemize}\item $\{u'_n\}_{n\in\N}$ is a bounded sequence in $L^\infty(0,T;BV(\Omega)_2^*)$ and therefore there is a net $\{u'_\alpha\}$ such that $$u'_\alpha\to \xi \ \in (L^1(0,T;BV(\Omega)_2))^* \quad {\rm weakly}^*.$$ \item There is $\z(t)\in X(\Omega)$ with $\|\z(t)\|_{\lio}\leq 1$ such that  $\z_n\to \z\in L^\infty(Q;\R^N)$ weakly$^*$. \item $\xi(t)=\div(\z(t) )$ in $\mathcal D'(\Omega)$ a.e. $t\in [0,T].$ \item $\xi$ is the time derivative of $u$ in the sense of Definition \ref{Def2}.\item $[\z(t),\nu]=\lim_{\alpha}[\z_\alpha(t),\nu]=1$,  $\mathcal H^{N-1}$-a.e. on $\partial\Omega$, a.e. $t\in [0,T]$. \item $\xi=\div(\z)$ in $(L^1(0,T;BV(\Omega)_2))^*$ in the sense of Definition \ref{Def3}. \end{itemize}

\medskip
\noindent{\it Step 3. Entropy inequality.}
Let $0\leq \eta\in \mathcal D(Q)$, $S\in\mathcal P$ and $h,l>0$. Since $u'_n(t)=\div(\z_n(t))$, we multiply this equation by $\eta(t)S(T_h(u_n(t)-T_h(l)))$. Integrating in $Q$ we get:

$$\int_0^T\int_\Omega \frac{d}{dt} j_{S,h,l}(u_n(t))\eta(t)= \int_0^T\int_\Omega \div(\z_n(t))S(T_h(u_n(t)-T_h(l)))\eta(t)$$$$=-\int_0^T\int_\Omega (\z_n(t), D(S(T_h(u_n(t)-T_h(l)))\eta(t))=-\int_0^T\int_\Omega \eta(t)|D(S(T_h(u_n(t)-T_h(l)))|$$$$-\int_0^T\int_\Omega S(T(T_h(u_n(t))-T_h(l)))\z_n(t)\cdot\nabla \eta(t)\,.$$
Therefore, $$\int_0^T\int_\Omega \eta(t)|D(S(T_h(u_n(t)-T_h(l)))|$$$$=-\int_0^T\int_\Omega S(T(T_h(u_n(t))-T_h(l)))\z_n(t)\cdot\nabla \eta(t)+\int_0^T\int_\Omega j_{S,h,l}(u_n(t))\eta_t(t)\,.$$
By the lower semicontinuity of the total variation, letting $n\to\infty$ we finally conclude that $$\int_0^T\int_\Omega \eta(t)|D(S(T_h(u(t)-T_h(l)))|$$$$\leq-\int_0^T\int_\Omega S(T(T_h(u(t))-T_h(l)))\z(t)\cdot\nabla \eta(t)+\int_0^T\int_\Omega j_{S,h,l}(u(t))\eta_t(t).$$

}

\medskip

{We give now a sketch of the proof of uniqueness which} follows closely the proof of Theorem \ref{uniqpent}.

 Suppose that $u=u(t,x)$, $\z=\z(t,x)$ and $v=v(s,y)$, $\overline \z=\z(s,y)$ and take $l_1=v(s,y)$, $S=T_k^+$ as the constant and the truncation in (\ref{entropylargetvf}) for $u$ and $l_2=u(t,x)$, $S=T_k^-$ for $v$. Then,

  $$\int_{Q} \eta |D_xT_k^+(T_h(u(t,x))-T_h(v(s,y)))|\, dt$$$$+\int_{Q} T_k^+(T_h(u(t,x))-T_h(v(s,y)))\z(t,x)\cdot\nabla_x \eta \,dx \,dt$$$$\leq\int_{Q} j_{T_k^+,h,v(s,y)}(u(t,x))\eta_t\, dx\, dt$$
  for all $k,h>0$, $0\leq\eta\in \mathcal D(Q)$ and
   $$\int_{Q} \eta |D_yT_k^-(T_h(v(s,y))-T_h(u(t,x)))| \, ds$$$$+\int_{Q} T_k^-(T_h(v(s,y))-T_h(u(t,x)))\overline \z(s,y)\cdot\nabla_y \eta \,dy \,ds$$$$\leq\int_{Q} j_{T_k^-,h,u(t,x)}(v(s,y))\eta_s\, dy\, ds$$ for all $k,h>0$, $0\leq\eta\in \mathcal D(Q)$.

   Let $0 \leq \phi \in
{\mathcal D}(0, {T})$, $0\leq \psi\in \mathcal D(\Omega)$, $\rho_m$ a classical sequence of mollifiers in $\Omega$ and $\tilde{\rho}_n$ a sequence of mollifiers in $\R$. Define
$$\eta_{m,n}(t, x, s, y):= \rho_m(x - y) \tilde{\rho}_n(t - s) \phi
\bigg(\frac{t+s}{2} \bigg)\psi\left(\frac{x+y}{2}\right).$$

Integrating the equations above in the other two variables and adding up both inequalities we get:
$$\int_{Q\times Q}\eta_{m,n}|D_x T_k^+(T_h(u)-T_h(v))|+\int_{Q\times Q}\eta_{m,n}|D_y T_k^+(T_h(u)-T_h(v))|$$$$+\int_{Q\times Q}(\z(t,x)-\z(s,y))\cdot(\nabla_x+\nabla_y)\eta_{m,n}T_k^+(T_h(u)-T_h(v))$$$$
+\int_{Q\times Q}\overline \z(s,y)\cdot\nabla_x \eta_{m,n}T_k^+(T_h(u)-T_h(v))-\int_{Q\times Q}\z(t,x)\cdot\nabla_y \eta_{m,n}T_k^+(T_h(u)-T_h(v))$$$$\leq \int_{Q\times Q}j_{T_k^+,h,v(s,y)}(u)(\eta_{m,n})_t+j_{T_k^-,h,u(x,t)}(v)(\eta_{m,n})_s$$

Now, by Green's formula, $$\int_{Q\times Q}\eta_{m,n}|D_x T_k^+(T_h(u)-T_h(v))|+\int_{Q\times Q}|D_y T_k^+(T_h(u)-T_h(v))|$$$$+\int_{Q\times Q}\overline \z(s,y)\cdot\nabla_x\eta_{n,m} T_k^+(T_h(u)-T_h(v))-\int_{Q\times Q}\z(t,x)\cdot\nabla_y \eta_{n,m} T_k^+(T_h(v)-T_h(u))$$$$
=\int_{Q\times Q}\eta_{m,n}|D_x T_k^+(T_h(u)-T_h(v))|+\int_{Q\times Q}\eta_{m,n}(\overline \z(s,y),D_x T_k^+(T_h(u)-T_h(v)))$$$$+\int_{Q\times Q}\eta_{m,n}|D_y T_k^+(T_h(u)-T_h(v))|-\int_{Q\times Q}\eta_{m,n}(\z(t,x),D_y T_k^+(T_h(u)-T_h(v)))\geq 0$$
Therefore,

$$\int_{Q\times Q}(\z(t,x)-\overline \z(s,y))\cdot(\nabla_x+\nabla_y)\eta_{m,n}T_k^+(T_h(u)-T_h(v))$$$$\leq \int_{Q\times Q}j_{T_k^+,h,v}(u)(\eta_{m,n})_t+j_{T_k^-,h,u}(v)(\eta_{m,n})_s$$

Passing to the limit when $h\to+\infty$ we get,
$$-\int_{Q\times Q}j_{T_k^+}(u(t,x)-v(s,y))((\eta_{m,n})_t+(\eta_{m,n})_s)$$$$+\int_{Q\times Q}(\overline \z(s,y)-\z(t,x))\cdot(\nabla_x+\nabla_y)\eta_{m,n}T_k^+(u-v)\leq 0\,,$$where $\displaystyle J_T(r):=\int_0^r T(s)\,ds$.
%
Passing to the limit when $n,m\to\infty$ yields:
\begin{equation}\label{123}-\int_{Q}j_{T_k^+}(u-v)\phi'(t)\psi(x)
 +\int_{Q}\phi(t)(\overline \z-\z)\cdot\nabla\psi(x)T_k^+(u-v)\leq 0\,.\end{equation}Now, working as in the proof of uniqueness of Theorem 1 in \cite{ACM}, and having in mind that $[\z(t),\nu]=[\overline \z(t),\nu]=1$ a.e. $t\in (0,T)$ and $\mathcal H^{N-1}-$a.e. on $\partial\Omega$ we can show that
  $$\lim_{\psi\uparrow {\chi_\Omega}}\int_{Q}\phi(t)(\overline \z-\z)\cdot\nabla\psi(x)T_k^+(u-v)\geq 0\ .$$ Then, letting $\psi\uparrow {\chi_\Omega}$ in \rife{123} we get
 $$-\int_{Q}j_{T_k^+}(u-v)\phi'(t)\leq 0\,,$$ for any $0\leq \phi\in {\mathcal D}(0,T)$. Therefore,
 $$\frac{\partial}{\partial t}\int_{Q}j_{T_k^+}(u-v)\leq 0\,,$$which implies
 $$\int_{\Omega}j_{T_k^+}(u-v)\leq \int_{\Omega}j_{T_k^+}(u_0-v_0)\,.$$Diving last expression by $k$ and letting $k\to 0$ we finally get
 $$\int_{\Omega}(u-v)^+\leq \int_{\Omega}(u_0-v_0)^+\,.$$

\end {proof}

\subsection{Some explicit examples of evolution}

In this section we give two explicit examples of evolution of large solutions for the total variation flow. In the first example, the initial data satisfies the boundary condition in a strong sense (i.e. $\exists \lim_{x\to\partial\Omega}u_0(x)=+\infty$ for all $x\in\partial\Omega$.) In this case, the large solution still verifies this condition.

\begin{example}
  Let us consider  $\Omega=B_1(0)\subset \R^2$ $$u_0(x):=\left\{
  \begin{array}{cc}  0 \quad & {\rm if \ } \|x\|\leq \frac{1}{2} \\
   \\ \log\left(\frac{\|x\|}{1-\|x\|}\right) & {\rm if \ } \frac{1}{2}\leq \|x\|\leq 1
   \end{array}\right.$$

 We look for a solution of \eqref{large1lp} of the form: \begin{equation}
  \label{explicit1}u(t,x)=a(t)\chi_{B_{r(t)}(0)}+b(t,\|x\|)\chi_{\Omega_{r(t)}}
\end{equation}with $\Omega_{r(t)}:=\Omega\setminus B_{r(t)}$,  $0<r(t)$ to be found and
such that $b$ is increasing with respect to its second variable and $a(t)=b(t,r(t))$ a.e $t\in ]0,T]$. {Note that in this case, $$\frac{Du}{|Du|}\chi_{\Omega_{r(t)}}=sign\left(\frac{\partial b}{\partial \|\cdot\|}(t,\|x\|)\right)\frac{x}{\|x\|}\chi_{\Omega_{r(t)}}=\frac{x}{\|x\|}\chi_{\Omega_{r(t)}}$$}

Then, we may define $$\z(t,x):=\left\{
\begin{array}
  {cc}\frac{x}{r(t)} & \quad {\rm if \ } \|x\|\leq r(t) \\ \\\frac{x}{\|x\|} & \quad {\rm if \ } r(t)\leq \|x\|<1
\end{array}
\right.$$ {in order to have that $[\z(t),\nu^{r(t)}]^-=1=[\z(t),\nu^{r(t)}]^+$, denoting respectively the interior and the exterior trace of the normal component of $\z$ with respect to the ball $B_{r(t)}$.}

Then, $${\rm div} \z(x,t)=\left\{\begin{array}
  {cc} \frac{2}{r(t)} & \quad {\rm if \ } \|x\|\leq r(t) \\ \\ \frac{1}{\|x\|} & \quad {\rm if \ } r(t)\leq \|x\|<1\end{array}\right.$$

  Therefore, $$\frac{1}{\|x\|}=b_t(t,\|x\|),$$ which implies that $$b(t,\|x\|)\chi_{\Omega_{r(t)}}=\left(\log\left(\frac{\|x\|}{1-\|x\|}\right)+\frac{t}{\|x\|}\right)\chi_{\Omega_{r(t)}}$$

  Then, $$\frac{2}{r(t)}=a'(t)=b_t(t,r(t))+\frac{\partial b}{\partial r}(t,r(t))r'(t)=\frac{1}{r(t)}+\left(\frac{1}{r(t)(1-r(t))}-\frac{t}{r^2(t)}\right)r'(t)$$Thus, $r(t)$ must solve the following ODE: $$\left(\frac{1}{1-r(t)}-\frac{t}{r(t)}\right)r'(t)=1;$$ with initial condition $r(0)=\frac{1}{2}$. Therefore, $$r(t)=\frac{\mathcal{ W}\left(-\frac{t+1}{2 e^{t+\frac{1}{2}}}\right)}{t+1}+1\,,$$ where $\mathcal{ W}$ is the Lambert W-function (see Figure \ref{figura}).

\begin{figure}[htc]\label{figura}
\includegraphics[width=6cm]{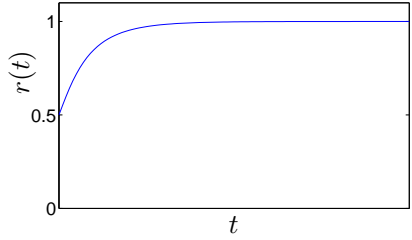}\caption{Evolution of the radius of the interior ball.}\end{figure}

In this case (see Figure \ref{figurasol}), $$u(t,x)=b(t,r(t))\chi_{B_{r(t)}(0)}+b(t,\|x\|)\chi_{\Omega_{r(t)}}$$
\begin{figure}[htc]\label{figurasol} \includegraphics[width=6cm]{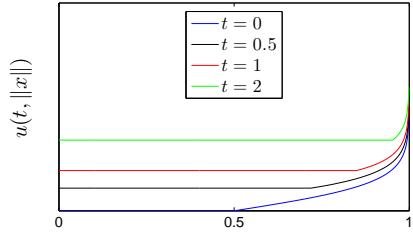}\caption{Large solution in radial coordinates.}\end{figure}

{ Then, $\frac{\partial b}{\partial \|\cdot\|}(t,\|x\|)>0$ and therefore} it is immediate to prove that $u\in L^1(0,T;W^{1,\infty}(\Omega))\cap W^{1,1}(0,T;L^1(\Omega))$, $\z(t)\in X(\Omega)$ a.e $t\in [0,T]$, $u'(t)=\div \z(t)$ a.e. $t\in (0,T)$ and $$\int_\Omega (\z(t),Dp(u(t)))=\int_\Omega |Dp(u(t))|\,,\quad \forall p\in\mathcal P\,,$$$$[\z(t),\nu]=1\,,\qquad {\rm for \ all\   } (t,x)\in [0,T]\times\partial\Omega .$$
From here, it follows that $u$ is the entropy solution of \eqref{large1lp} with $u_0$ as initial data.


\end{example}
{
Second example is totally different in nature. In this case, the initial data is $u_0=0$ and in it it is shown the influence of the domain on the solution. We point out that, as proved in Theorem \ref{theoremp1}, the entropy large solution is uniformly bounded for all $t\in [0,T]$ and then, the boundary condition is fulfilled in the weak sense (\ref{boundary}). Note also that, in general, the solution is a genuine BV-function since it may have jump discontinuities}.
\begin{example}
  Let $C$ be a bounded convex subset of $\R^N$  of class $C^{1,1}$ (in particular, it verifies the uniform interior ball condition). We focus on the following problem:
  \begin{equation}
  \label{large1lpcal}\left\{\begin{array}
    {cc}\displaystyle \frac{\partial u}{\partial t}=\div\left(\frac{Du}{|Du|}\right) & \quad {\rm in \ } (0,T)\times C\\ \\ u(t,x)=+\infty & \quad {\rm on \ } (0,T)\times\partial C \\ \\ u(0,x)=0 & \quad {\rm in \ } C
  \end{array}\right.
\end{equation}

We need to recall the approach and several results given in \cite{acc} which we gather together in the next result:


\begin{theorem}[\cite{acc}]Consider the problem $$(P)_\lambda:=\min_{F\subseteq C} Per(F)-\lambda|F|$$
  Then, there is a convex set $K\subseteq C$ (the Cheeger set, see \cite{ac,ccn} for details) which is a solution of $(P)_{\lambda_K}$ with $\lambda_D:=\frac{Per(D)}{|D|}$  for any $D\subseteq C$. For any $\lambda>\lambda_K$ there is a unique minimizer $C_\lambda$ of $(P)_\lambda$ and the function $\lambda\to C_\lambda$ is increasing and continuous. Moreover, $C_\mu=C$ iff $\mu\geq\max\{\lambda_C,(N-1)\|\mathbb H_C\|_\infty\}$ with $\mathbb H_C(x)$ being the mean curvature of $\partial C$ at the point $x$.
\end{theorem}

Let $K$ be the Cheeger set contained in $C$ defined in the previous result. For each $\lambda\in (0,+\infty)$ let $C_\lambda$ be the minimizer of problem $(P)_\lambda$. We take $C_\lambda=\emptyset$ for any $\lambda<\lambda_K$. Using the monotonicity of $C_\lambda$ and the fact that $|C\setminus\cap\{C_\lambda : \lambda>0\}|=0$ we may define $$H_C(x):=\inf\{\lambda : x\in C_\lambda\}$$


In Theorem 17 in \cite{acc} it is shown that $$v(t,x)=(1-H_C(x)t)^+\chi_C$$ is the entropy solution for the Cauchy problem for the Total Variation with $v_0=\chi_C$ as initial data. In particular, it is obtained a vector field $\xi_C\in X(\R^N)$ with $\|\xi_C\|_{   L^{\infty}(\rn)  }\leq 1$ such that $\div\xi_C=-H_C\chi_C$ in $\mathcal D'(\R^N)$, $$\int_{\R^N} (\xi_C,Dv)=\int_{\R^N}|Dv|$$ and $$[\xi_C,\nu^C]=-1\quad\mathcal H^{N-1}-{\rm a.e. \ on }\ \partial\Omega$$Taking $\z:=-\xi_C\chi_C$ and with the same proof as in \cite{acc}, we can show that $$u(t,x)=H_C(x)t$$ is the entropy large solution of \rife{large1lpcal}. Therefore, the speed of the growth  of the large entropy solution is the speed of decrease of the solution of the corresponding Cauchy problem with $\chi_C$ as initial datum.
\end{example}

{
\begin{remark}
 Let us finally observe that,  in the especial case that $C$ is calibrable (i.e. the Cheeger set $K$ coincides with $C$), then the large solution of (\ref{large1lpcal}) is exactly: $$u(t,x)=\frac{Per(C)}{|C|}t\,.$$
\end{remark}
}




\begin{thebibliography}{00}
%
%

\bibitem{nv}  Al Sayed, W.,  Veron, L. (2009). {On uniqueness of large solutions of nonlinear parabolic equations in nonsmooth domains}.  {\it Adv. Nonlinear Stud.} {9}:   149--164.

\bibitem{ac} Alter F.,  Caselles, V. (2009). \newblock{Uniqueness of the Cheeger set of a convex body}. {\it Nonlinear Analysis}.
 {70(1)}: 32--44


\bibitem{acc} Alter, F.,  Caselles, V.,  Chambolle, A. (2005). \newblock {A characterization of convex calibrable sets in $\R^N$}.
\newblock  {\it Math. Ann.} {332}: 329--366.

\bibitem{Ambrosio}
Ambrosio, L.,  Fusco, N., and  Pallara, D.  (2000).
\newblock {\it Functions of Bounded Variation
and Free Discontinuity Problems}, \newblock Oxford Mathematical Monographs.

\bibitem{ACM} Andreu, F.,  Caselles, V.,  Maz\'on, J.M. (2001). The Dirichlet Problem for the Total Variation Flow. {\it J. Funct. Analysis} {180}:347--403.

\bibitem{ACM4:01}
Andreu, F.,  Caselles, V.,  Maz\'on, J.M. (2002).
\newblock {Existence and uniqueness of solution for a parabolic quasilinear
  problem for linear growth functionals with {$L^1$} data}.
\newblock  {\it Math. Ann.} {\bf 322}: 139--206.

\bibitem{ACMBook}
Andreu, F.,  Caselles, V.,  Maz\'on, J.M. (2004).
\newblock {\it Parabolic Quasilinear Equations Minimizing
 Linear Growth Functionals}.
\newblock   Progress in Mathematics:  Birkh\"auser.

\bibitem{Anzellotti1}
Anzellotti, G. (1983).
\newblock {Pairings Between Measures and Bounded Functions
and Compensated Compactness}.
\newblock {\it Ann. di Matematica Pura ed Appl.} 135:293--318.

\bibitem{B6}  B\'enilan, Ph., Boccardo, L.,  Gallou\"et, T.,  Gariepy, R.,  Pierre, M.,   Vazquez, J.L. (1995).  An $L^{1}-$theory of existence and uniqueness of nonlinear elliptic equations. {\it Ann. Scuola Norm. Sup. Pisa Cl. Sci.} {22}: 241--273.

\bibitem{becr}  B\'enilan, Ph.,   Crandall, M. G. (1991). Completely accretive operators, in Semigroups Theory
and Evolution Equations.  Ph. Clement et al., Eds. 41--76: Dekker, New York.

%
%


\bibitem{BlMu}
 Blanchard, D.,  Murat, F., (1997). Renormalized solutions of nonlinear
parabolic problems with $L^1$ data: existence and uniqueness, {\it Proc. Royal Soc.  Edinburgh Section A.} {127}:  1137--1152.




%
%

%
%
%
%
%
%
%
%
%


\bibitem{bmp}Boccardo, L.,   Murat, F.,  Puel, J.P. (1992). {$L\sp \infty$
estimate for some nonlinear elliptic partial differential
equations and application to an existence result}. {\it SIAM J. Math.
Anal.} { 23}:326--333.

\bibitem{BIV}  Bonforte, M., Iagar, R. G.,  V\'azquez, J. L. (2010). Local smoothing effects, positivity and Harnack inequalities for the fast $p-$Laplacian equation. {\it Adv. Math.} 224: 2151--2215.


\bibitem{cw}  Carrillo, J.,  Wittbold, P. (1999). Uniqueness of Renormalized Solutions of Degenerate Elliptic-Parabolic Problems. {\it J. Differential Equations}. {156}:93--121.


\bibitem{ccn}  Caselles, V.,  Chambolle, A.,  Novaga, M. (2007). Uniqueness of the Cheeger set of a convex body, {\it Pacific J. Math.} {232 (1)}:77--90.

\bibitem{cv}  Chasseigne, E., Vazquez,   J. L. (2002). Theory of extended solutions for fast-diffusion equations in optimal classes of data. Radiation from singularities. {\it Arch. Ration. Mech. Anal.} {164}:133--187.

\bibitem{crli}    Crandall, M. G.,  Liggett, T. M. (1971). Generation of semigroups of nonlinear transformations
on general Banach spaces.  {\it Amer. J. Math.} 93: 265--298.

\bibitem{dl}  Diaz, G.,  Letelier, R. (1993). Explosive solutions of quasilinear elliptic equations: existence and uniqueness. {\it Nonlinear Anal.} {20}:97--125.


\bibitem{dib}  DiBenedetto, E. (1993).    \emph{Degenerate Parabolic Equations}, Universitext: Springer-Verlag, New York.

 \bibitem{DU}
 Diestel, J.,  Uhl, Jr.,  J.J. (1977).
\newblock {\it Vector Measures},
\newblock Math. Surveys {15}: Amer. Math. Soc., Providence.


\bibitem{dpp} Droniou, J.,  Porretta,  A.,  Prignet, A. (2003). Parabolic capacity and soft measures for nonlinear equations. {\it Potential
Anal. } {19}: 99--161.

    \bibitem{EG}
 Evans L. C.,   Gariepy, R.F. (1992).
\newblock {\it Measure Theory and Fine Properties of Functions}.
\newblock Studies in Advanced Math: CRC Press.

\bibitem{ke} Keller,   J. B. (1957). {On solutions of $\Delta u=f(u)$}.
{\it Commun. Pure Appl. Math.} {10}:503--510.

\bibitem{lan}  Landes, R. (1981). {On the existence of weak solutions for quasilinear parabolic boundary value problems}.  {\it Proc. Royal Soc. Edinburgh Sect. A}. {89}: 217--237.

\bibitem{ll}  Lasry, J.-M.,  Lions,  P.-L. (1989). {Nonlinear elliptic  equations with singular boundary conditions and stochastic control with state constraints. I.  The model problem}.   {\it Math. Ann.} {283}:583--630.

\bibitem{lepe}  Leoni, F.,  Pellacci,  B. (2006). Local estimates and global existence for strongly nonlinear parabolic equations with locally integrable data. {\it J. Evol. Equ.} 6:113--144.

\bibitem{leo} Leonori, T. (2007). Large solutions for a class of nonlinear elliptic equations with gradient terms. {\it Adv. Nonlinear Stud.} 7:237--269.

\bibitem{lp}  Leonori, T.,    Petitta, F.  (2011) Local estimates for  parabolic equations with nonlinear gradient terms.  {\it Calc.  Var.  and PDE's} {42} (1): 153--187.

\bibitem{lepo}Leonori, T.,    Porretta,  A. (2007/08). The boundary behavior of blow-up solutions related to a stochastic control problem with state constraint.  {\it  SIAM J. Math. Anal.} 39:1295--1327.



%


\bibitem{os}   Osserman, R. (1957). {On the inequality $\Delta u\geq f(u)$},  {\it Pacific J. Math.} {7}:1641--1647.

 \bibitem{pe}  Petitta, F. (2008).  Renormalized solutions of nonlinear parabolic equations with general measure data, {\it Ann. Mat. Pura Appl.} {187}:563--604.


%
\bibitem{po}
 Porretta, A. (1999). Existence results for nonlinear parabolic equations via strong convergence of truncations, {\it Ann. Mat. Pura Appl. } {177}:143--172.

\bibitem{Schwartz}
 Schwartz,  L. (1974/75).
\newblock {\it Fonctions mesurables et $^*$-scalairement mesurables,
mesures banachiques major\'ees, martingales banachiques, et
propi\'et\'e de Radon-Nikod\'ym}.
\newblock S\'em. Maurey-Schawartz:  Ecole Polytech. Centre de
Math.

\bibitem{si}  Simon, J. (1987). {Compact sets in the space $L^p(0,T;B)$}.  {\it Ann. Mat. Pura Appl.} {146}:65--96.

\bibitem{Ziemer}
 Ziemer, W.P. (1989).
\newblock {\it Weakly Differentiable Functions}.
\newblock GTM 120: Springer Verlag.

\end{thebibliography}
\end{document}